\documentclass{amsart}

\usepackage{amsmath,amsthm,amsfonts}
\usepackage{latexsym}
\usepackage{amssymb}
\usepackage{pstricks}

\newtheorem{thm}{Theorem}[section]
\newtheorem{lemma}[thm]{Lemma}
\newtheorem{cor}[thm]{Corollary}
\newtheorem{prop}[thm]{Proposition}

\def\C{{\mathbb C}}
\def\del{\delta}
\def\sign{{\rm sign}}
\def\less{\lesssim}
\def\la{\langle}
\def\ra{\rangle}
\def\ad{{\rm ad}}
\def\Re{{\,\rm Re\,}}
\def\Im{{\,\rm Im\,}}

\def\Ai{{\rm Ai}}
\def\Bi{{\rm Bi}}
\def\eps{{\varepsilon}}
\def\les{\lesssim}

\def\R{{\mathbb R}}
\def\nn{\nonumber}
\def\calH{\mathcal H}
\def\del{\partial}
\def\rem{{\tt rem}}
\def\prefac{\frac{i}{\hbar}}

\def\snabla{ {\slash\!\!\! \nabla} }
\def\xphot{x_{{\rm max}}}

\def\vphi{\varphi}
\def\lan{\langle}
\def\ran{\rangle}
\def\supp{\mathrm{supp}}

\numberwithin{equation}{section}

\begin{document}

\title[Decay of linear waves on Schwarzschild]{On pointwise decay of linear waves on a Schwarzschild black hole background}

\author{Roland Donninger}
\address{University of Chicago, Department of Mathematics,
5734 South University Avenue, Chicago, IL 60637, U.S.A.}
\email{roland.donninger@epfl.ch}
\thanks{The first author
is an Erwin Schr\"odinger Fellow of the
FWF (Austrian Science Fund) Project No. J2843 and
he wants to thank Peter C. Aichelburg for his support and Piotr Bizo\'n for helpful discussions.}

\author{Wilhelm Schlag}
\address{University of Chicago, Department of Mathematics,
5734 South University Avenue, Chicago, IL 60637, U.S.A.}
\email{schlag@math.uchicago.edu}
\thanks{The second author was partly supported by the National
Science Foundation DMS-0617854 and a Guggenheim fellowship.}

\author{Avy Soffer}
\address{Rutgers University, Department of Mathematics, 110 Freylinghuysen Road, Piscataway, NJ 08854, U.S.A.}
\email{soffer@math.rutgers.edu}
\thanks{The third author wants to thank A. Ori and T. Damour for helpful
discussions, the
IHES France for the invitation and the NSF DMS-0903651 for partial support.}

\begin{abstract}
We prove sharp pointwise $t^{-3}$ decay for scalar linear perturbations of a Schwarzschild black hole without
symmetry assumptions on the data. We also consider electromagnetic and gravitational perturbations for which we obtain decay rates $t^{-4}$, and $t^{-6}$, respectively. We proceed by
decomposition into angular momentum~$\ell$ and summation of the  decay estimates on the Regge-Wheeler equation for fixed~$\ell$. We encounter a  dichotomy: the decay
law in time is entirely determined by the asymptotic behavior of the Regge-Wheeler potential in the far field, whereas the growth of the
constants in~$\ell$ is dictated by the behavior of the Regge-Wheeler potential in a small neighborhood around its maximum. 
In other words, the tails are controlled by small energies, whereas the number of angular derivatives needed on the data is determined by 
energies close to the top of the Regge-Wheeler potential.  This dichotomy corresponds to the well-known principle that for  initial times the decay  reflects the presence of
complex resonances generated by the potential maximum, whereas for later times the tails are determined by the far field. However, we do not invoke complex resonances at all, but rely
instead on semiclassical Sigal-Soffer type propagation estimates based on a Mourre bound near the top energy. 
\end{abstract}

\maketitle

\section{Introduction}\label{sec:Prelim}

The study of linear waves on fixed black hole backgrounds has a long history in
mathematical relativity and very recently, major progress has been made on various aspects of the 
problem, see, e.g., \cite{Tataru09}, \cite{DSS}, \cite{Luk10}, \cite{Luk10b}, \cite{DafRod08},
\cite{DafRod08}, \cite{DafRod10}, \cite{AB09}, \cite{MTT11}, \cite{MMTT10},
\cite{T09}, \cite{TT11}, \cite{FKSY06}
to name just a few of the more recent contributions.
We refer the reader to the excellent lecture notes by Dafermos and Rodnianski 
\cite{DafRod3} 
for the necessary background and a more detailed list of references.
Understanding the behavior of linear waves on fixed backgrounds is supposed to be
a necessary prerequisite for the study of the stability of black hole spacetimes in full 
general relativity, one of the major open problems in the field.
The goal of this paper is to prove point-wise in time decay estimates
for linear waves on the background of a Schwarzschild black hole. To be precise, let 
$$g=-\left (1-\frac{2M}{r}\right )dt^2+\left (1-\frac{2M}{r}\right )^{-1}dr^2+r^2(d\theta^2
+\sin^2 \theta\; 
d\varphi^2)$$ be
the Schwarzschild metric on $(t,r,\theta, \varphi)\in 
\R\times (2M,\infty)\times (0,\pi) \times (0,2\pi)$. Introducing the tortoise
coordinate
\[
 x=r+2M\log(\frac{r}{2M}-1)
\]
reduces the wave equation $\Box_g \psi=0$ to the form
\begin{equation}
 \label{eq:wave}
-\del_t^2\psi +\del_x^2 \psi -\frac{F}{r}\frac{dF}{dr}\psi + \frac{F}{r^2}\Delta_{S^2}\psi=0
\end{equation}
where $F=\frac{dr}{dx}$.
Our main result is as follows:

\begin{thm}
 \label{thm:main}
The following decay estimates hold for solutions $\psi$ of~\eqref{eq:wave} with data $\psi[0]=(\psi_0,\psi_1)$:
\begin{align}
 \| \la x\ra^{-\frac92-}  \psi(t) \|_{L^2} &\les \la t\ra^{-3} \| \la x\ra^{\frac92+} (\snabla^5 \del_x\psi_0, \snabla^5 \psi_0, \snabla^4\psi_1) \|_{L^2} \label{eq:decaywaveL2}  \\
\| \la x\ra^{-4}  \psi(t) \|_{L^\infty} &\les \la t\ra^{-3} \| \la x\ra^4 (\snabla^{10} \del_x\psi_0, \snabla^{10} \psi_0, \snabla^{9} \psi_1) \|_{L^1} \label{eq:decaywaveL1}
\end{align}
where $\snabla$ stands for the angular derivatives\footnote{The notation $a\pm$ stands for $a\pm\eps$ where $\eps>0$ is arbitrary (the choice determines the constants involved). Also, instead
of $(\snabla^{10},\snabla^9)$ in~\eqref{eq:decaywaveL1} one needs less, namely $(\snabla^{\sigma+1},\snabla^{\sigma})$ where $\sigma>8$ is arbitrary, see the proof in Section~\ref{sec:proof}
for details.}. Here $L^p:=L^p_x(\R;L^p(S^2))$
and $\langle x \rangle:=(1+|x|^2)^{1/2}$. 
\end{thm}

The rate~$t^{-3}$ is well-known to be optimal for radial data, i.e., vanishing angular momentum, see for example~\cite{DS}. 
The same applies to the weight~$\la x\ra^{-4}$. 
We remark that Tataru~\cite{Tataru09} has recently obtained a striking result of this flavor but for essentially smooth data (he apparently needs a large number of angular derivatives).
On the other hand, he derives his result in the greater generality of a Kerr background (for small parameter $a$) and also obtains a 
Huygens principle. 
We expect that our methods can be generalized
to cover these as well as other scenarios, but we do not pursue this here.  Another result in this direction, albeit for Schwarzschild de-Sitter,
is due to Bony and H\"afner~\cite{BH}. By means of a resonance expansion they prove local exponential decay in that setting for compactly supported data. 

Let us mention two (related) extensions of Theorem~\ref{thm:main}. The first extension
concerns the type of black hole perturbation we can cover. As stated above, Theorem~\ref{thm:main} applies to scalar perturbations.
However, one has similar statements (but with {\em better decay}, see below) for gravitational and electromagnetic perturbations of the Schwarzschild black hole which appear as $\sigma=-3$ and $\sigma=0$,
respectively, in the Regge-Wheeler potential, see~\eqref{eq:ReggeWheeler} below.  In the case of $\sigma=-3$ one needs the data to be perpendicular 
to the spherical harmonics $Y_0$ and $\{ Y_{\ell,1}\}_{\ell=-1}^1$, and for $\sigma=0$ one needs to require orthogonality of the data to $Y_{0}$. 
These  conditions  eliminate  a gauge freedom inherent in the problem (such as changing the mass or the charge)\footnote{From the point of view of the decay estimates in~\cite{DSS}, these values need to be excluded as they are precisely the ones that give rise to a zero energy resonance.}.  We can cover these other values of~$\sigma$ for two reasons: (i) the decay bounds in~\cite{DSS} apply to them, and (ii) the WKB analysis in Section~\ref{sec:low} which is the only place where $\sigma$ plays a role in this paper, is insensitive to this modification.

The second extension concerns faster rates of decay. In fact, 
Theorem~\ref{thm:main} actually gives an arbitrary rate of decay, i.e., $t^{-N}$ for any $N$, provided
the data are perpendicular to the first few spherical harmonics (the exact number depending on~$N$). This follows immediately by inspection of our proof, since~\cite{DSS} establishes
accelerated rates as in Price's law~\cite{Price1}, \cite{Price2} for a fixed spherical harmonic.  One formulation of this result reads  as follows:

\begin{thm}
 \label{thm:main2}
 Suppose that $\psi[0]\perp Y_j$ where $Y_j$ are the spherical harmonics on~$S^2$ with eigenvalues less than~$\ell(\ell+1)$ with $\ell>0$.
  Then one has the following faster rates of decay
 for solutions $\psi$ of~\eqref{eq:wave} with data $\psi[0]=(\psi_0,\psi_1)$:
\begin{align}
\| \la x\ra^{-m}  \psi(t) \|_\infty &\les \, \la t\ra^{-(2\ell+2)} \| \la x\ra^{m} (\snabla^{n+1} \del_x\psi_0, \snabla^{n+1} \psi_0, \snabla^{n}\psi_1)  \|_1 \label{eq:decaywaveL1_acc}
\end{align}
The implicit constant depends on $\ell$ and $n,m$ are sufficiently large integers which grow linearly in~$\ell$. 
\end{thm}

The decay predicted by Price's law is $t^{-2\ell-3}$ but at the moment we only obtain~$t^{-2\ell-2}$, see~\cite{DSS}.  In particular, for gravitational perturbations we take $\ell=2$ 
and for electromagnetic ones~$\ell=1$ leading to the decay rates $t^{-6}$ and $t^{-4}$, respectively, as stated in the abstract. Note that according to Price's law one should have  $t^{-7}$ and $t^{-5}$, respectively. 

\subsection{Extension to more general data}
As stated,  
Theorems \ref{thm:main} and \ref{thm:main2} require 
the initial data to vanish at the bifurcation sphere
$x \to -\infty$.
This is clearly a disadvantage of the result from the physical point of view since one would like
to cover more general perturbations.
However, there exists a classical construction by Kay and Wald \cite{KW87} which enables one to overcome 
this restriction.
In order to explain this clever geometric argument, we have to briefly digress into some more advanced 
aspects of the Schwarzschild geometry.
As is well-known, the Schwarzschild coordinates $(t,r, \theta, \varphi)$
cover only a small portion of a bigger manifold
which is referred to as \emph{maximally extended Schwarzschild} or the 
\emph{Kruskal extension}, see, e.g., \cite{HE73}, \cite{Wald84}.
This is shown by introducing a new coordinate system $(T,R,\theta,\varphi)$ which is related to
the Schwarzschild coordinates by
$$ R^2-T^2=\left(\frac{r}{2M}-1\right )e^{r/(2M)},\quad t=2M \log \left(\frac{R+T}{R-T}\right). $$
In Kruskal coordinates the Schwarzschild metric reads
$$ g=\frac{32M^3}{r}(-dT^2+dR^2)+r^2(d\theta^2+\sin^2\theta\;d\varphi^2) $$
for $R>|T|$ and $r$ is now interpreted as a function of $T$ and $R$.
However, the singularity at $r=2M$ (which corresponds to $R=|T|$) has disappeared and nothing 
prevents us from allowing all values of $T$ and $R$ provided that $R^2-T^2>-1$.
This yields the celebrated Kruskal extension.
A spacetime diagram of the Kruskal extension
is depicted in Fig.~\ref{fig:kruskal} and the wedge $\mathcal{S}$ (which consists of the two
shaded regions in Fig.~\ref{fig:kruskal}) represents the original Schwarzschild
manifold.
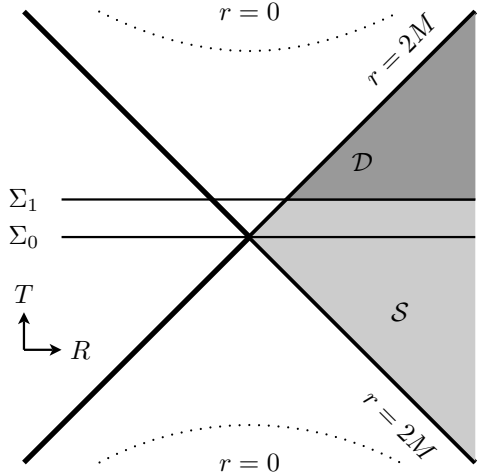
\begin{figure}[ht]
\ifx\JPicScale\undefined\def\JPicScale{1}\fi
\psset{unit=\JPicScale mm}
\psset{linewidth=0.3,dotsep=1,hatchwidth=0.3,hatchsep=1.5,shadowsize=1,dimen=middle}
\psset{dotsize=0.7 2.5,dotscale=1 1,fillcolor=black}
\psset{arrowsize=1 2,arrowlength=1,arrowinset=0.25,tbarsize=0.7 5,bracketlength=0.15,rbracketlength=0.15}
\begin{pspicture}(0,0)(72.5,60)
\psline[linewidth=0.7](0,60)(60,0)
\psline[linewidth=0.7,fillstyle=vlines](0,0)(60,60)
\rput(0,30){$\Sigma_0$}
\rput(0,35){$\Sigma_1$}
\rput{45}(50,55){$r=2M$}
\rput{-45}(50,5){$r=2M$}
\newrgbcolor{userFillColour}{0.8 0.8 0.8}
\psline[fillcolor=userFillColour,fillstyle=solid](60,60)
(30,30)(60,0)
\psline(5,30)(60,30)
\psline(5,35)(60,35)
\newrgbcolor{userFillColour}{0.6 0.6 0.6}
\psline[fillcolor=userFillColour,fillstyle=solid](60,60)
(35,35)(60,35)
\rput(45,40){$\mathcal{D}$}
\rput(50,20){$\mathcal{S}$}
\rput{0}(30,-37.5){\psellipticarc[linestyle=dotted](0,0)(42.5,42.5){61.93}{118.07}}
\psbezier[linestyle=dotted](10,60)(24,53)(36,53)(50,60)
\rput(30,60){$r=0$}
\rput(30,0){$r=0$}
\psline{<-}(0,20)(0,15)
\psline{->}(0,15)(5,15)
\rput(7.5,15){$R$}
\rput(0,22.5){$T$}
\end{pspicture}
\caption{The Kruskal spacetime. 
The two shaded regions together comprise the original exterior Schwarzschild manifold, denoted
by $\mathcal{S}$. 
The solution of the wave equation $\Box_g \psi=0$ in the darker shaded region $\mathcal{D}$ is
uniquely determined by Cauchy data on $\Sigma_1 \cap \mathcal{S}$.}
\label{fig:kruskal}
\end{figure}
The Kruskal spacetime is globally hyperbolic and in order to describe the Kay-Wald argument, 
we consider two Cauchy surfaces 
$\Sigma_0$ and $\Sigma_1$ at $T=0$
and some small $T>0$, respectively.
The intersection $\Sigma_0 \cap \mathcal{S}$ corresponds
to the initial surface $t=0$ in Theorems \ref{thm:main} and \ref{thm:main2}.
Suppose now we prescribe initial data on $\Sigma_0$ (sufficiently regular, with sufficient decay at
spatial infinity but not necessarily vanishing at the bifurcation sphere $T=R=0$) and consider
the wave equation $\Box_g \psi=0$ with these data.
We are interested in the future development in the original Schwarzschild wedge $\mathcal{S}$.
According to the domain of dependence property of the wave equation (see \cite{HE73}) 
the development to the future of $\Sigma_1$ in $\mathcal{S}$ (the domain $\mathcal{D}$ in
Fig.~\ref{fig:kruskal}) 
is entirely determined by the 
values of $\psi$ and $\psi_T$ on $\Sigma_1\cap \mathcal{S}$.
Now we prescribe initial data $(f,g)$ on $\Sigma_1$ such that
\begin{enumerate}
\item $f(R,\theta,\varphi)=-f(-R,\theta,\varphi)$, $g(R,\theta,\varphi)=-g(-R, \theta, \varphi)$,
\item $(f,g)$ coincide with $(\psi, \psi_T)$ on $\Sigma_1 \cap \mathcal{S}$,
\item $(f,g)$ are as regular as $(\psi, \psi_T)$ on $\Sigma_1$.
\end{enumerate}
It is obvious from the spacetime diagram Fig.~\ref{fig:kruskal} that this can be done.
Then we consider the solution $\tilde{\psi}$ of $\Box_g \tilde{\psi}=0$ with data $(f,g)$ 
on $\Sigma_1$.
By the aforementioned domain of dependence property we have $\tilde{\psi}=\psi$ in $\mathcal{D}$.
The key observation now is the existence of the discrete isometry $(T,R,\theta,\varphi)\mapsto
(T,-R,\theta,\varphi)$ which leaves the line $R=0$ invariant and guarantees that 
property (1) of the data $(f,g)$ is propagated by the wave flow, i.e., we have
$\tilde{\psi}(T,R,\theta,\varphi)=-\tilde{\psi}(T,-R,\theta,\varphi)$ which in particular implies
$\tilde{\psi}(T,0,\theta,\varphi)=0$ for all $T$.
As a consequence, by evaluating $\tilde{\psi}$ and $\tilde{\psi}_T$ on $\Sigma_0$,
we obtain new initial data on $\Sigma_0$ \emph{which vanish at the
bifurcation sphere and lead to the same solution in $\mathcal{D}$} as the original data
$(\psi,\psi_T)|_{\Sigma_0}$.
If the data are sufficiently regular, they have to vanish at least linearly in $R$ at the bifurcation
sphere which yields exponential decay with respect to the tortoise coordinate $x$ as $x\to -\infty$ and 
our Theorems \ref{thm:main} and \ref{thm:main2} apply.

We remark in passing that the discrete isometry which lies at the heart of the Kay-Wald argument
is a very fragile property 
which cannot be expected to hold in more general spacetimes.
Recently, Dafermos and Rodnianski \cite{DafRod09} 
devised a more robust method based on vector field multipliers
which is capable of extending decay estimates
up to the horizon.

\subsection{Strategy of proof of Theorems \ref{thm:main} and \ref{thm:main2}}

The strategy for the proof of Theorem~\ref{thm:main} is to decompose the solution  into spherical harmonics and then to sum the resulting decay estimates.
The wave equation~\eqref{eq:wave} at fixed angular momentum turns into a wave equation in $1+1$
dimensions, namely in the time variable~$t$ and the ``radial'' variable~$x$.
The angular derivatives in the estimates \eqref{eq:decaywaveL2} and~\eqref{eq:decaywaveL1} then arise as weights given by powers of the angular momentum.
 This procedure is not expected to yield the optimal bounds as far as the number of angular derivatives
 is concerned. The specific numbers appearing in  \eqref{eq:decaywaveL2}  above
 are a result of the Mourre estimate approach to the ``top of the barrier scattering'' which we develop
 in this paper. This Mourre estimate is non-classical in the sense that it needs to take into account that the top energy is trapping. We deal with the issue by means
 of the Heisenberg uncertainty principle (or the ground state of the semiclassical harmonic oscillator). The transition from our Mourre estimate to the decay in time
 is accomplished by means of Sigal-Soffer propagation theory going back to~\cite{SigSof2}, but the implementation we follow is~\cite{HSS}. 
 The further losses in terms of angular derivatives in~\eqref{eq:decaywaveL1} as compared to~\eqref{eq:decaywaveL2}  are due to  the Bernstein inequality and the $L^\infty$ bound on
 the spherical harmonics. 
 
 It is likely that
 Theorem~\ref{thm:main}  can be improved with regard to the number of angular derivatives required by a more detailed analysis of the spectral
 measure (at fixed angular momentum) for energies near the potential maximum. This would involve a reduction
 to Weber's equation and an explicit perturbative analysis of the Jost solutions instead of the more indirect Mourre-Sigal-Soffer method.
 However, since it would complicate this paper
 we have chosen not to follow that route. We emphasize that the number of derivatives~$\snabla$ appearing in our
 theorems is determined {\em exclusively} by the analysis near the maximum energy. 
 There is a sizable literature on the topic of scattering near
 a potential maximum and, more generally,  on scattering in the presence of trapping energies, see for example \cite{ABR2}, \cite{BFRZ1}, \cite{BCD}, \cite{GG},
 \cite{HS2}, \cite{Nak}, \cite{Ram}, and \cite{Sjo} and the references in these papers. However, we find that the available results in that direction are either
 not sharp enough for our purposes, or actually can be sidestepped completely with the Mourre approach we follow here. 

 For energies near zero, the Mourre estimates become degenerate.
 Therefore we need to rely on a WKB-type analysis of a semiclassical problem which we describe in 
 detail in Section~\ref{subsec:decomp} (the semiclassical parameter being $\hbar=\ell^{-1}$ where $\ell$ is the angular momentum).
 The main issue for low lying energies is that the errors of the
 perturbative analysis of the spectral measure (and the Jost solutions) have to be  controlled {\em simultaneously
 for all energies near zero and all small~$\hbar$}.  This was accomplished in~\cite{CSST} and~\cite{CST}. 

We remark that the technical part of this paper is entirely devoted
to large angular momenta - in other words, to the analysis of the
semiclassical equation. In fact, for angular momenta $0\le\ell\le
\ell_0$ where $\ell_0$ is large we invoke the bounds from~\cite{DSS}
and~\cite{DS}. The constants appearing in the decay bounds in
these papers grow rapidly in the angular momentum (in some
super-exponential fashion). This precludes us from summing them
in~$\ell$ and necessitates the separate WKB/Mourre analysis of this
paper. However, since the latter only applies to large~$\ell$ the
methods developed in~\cite{DSS} and~\cite{DS} are of crucial
importance for Theorem~\ref{thm:main}.

This paper is not self-contained, as it relies on the results
of~\cite{CSST}, \cite{CST}, \cite{DSS} and~\cite{DS}. Needless to say, there is
a long history concerning Price's law, see~\cite{Tataru09} as well as~\cite{DafRod3}. We refer
to those papers as well as to the introduction of our earlier
paper~\cite{DSS} for a detailed list of references and more
background. We would also like to mention that this paper as well as \cite{CSST}, \cite{CST}, \cite{DSS} and~\cite{DS},
are a result of those investigations into decay of wave equations on curved backgrounds with trapping metrics 
which began with the  surface of revolution papers~\cite{SSS1}, \cite{SSS2}. 

\subsection{Angular momentum decomposition}\label{subsec:decomp}

Restricting to spherical harmonics, the wave equation~\eqref{eq:wave} takes the form
\begin{equation} \label{eq:ellwave}
\del^2_t \psi-\del_x^2 \psi + V_{\ell,\sigma}(x)\psi=\del_t^2 \psi + \calH_{\ell,\sigma}\psi =0
\end{equation}
with the Regge-Wheeler potential
\begin{equation}\label{eq:ReggeWheeler}
V_{\ell,\sigma}(x)=\left (1-\frac{2M}{r(x)} \right )
\left (\frac{\ell(\ell+1)}{r^2(x)}+\frac{2M\sigma}{r^3(x)} \right )
\end{equation}
for $\sigma=1$. However, as mentioned before we allow for other values of $\sigma$ as well with the
physically relevant ones being $\sigma=-3,0$. 
We will take $\ell$ large and study the semiclassical operator
\begin{equation}
H=H(\hbar):= -\hbar^2 \del_x^2 + V(x;\hbar)
\end{equation}
with the normalization $V(x;\hbar):=\hbar^2 V_{\ell,\sigma}(x)$, and
$V(\xphot;\hbar)=1$ where $\xphot$ is the location of the unique maximum of the potential. 
Thus $\hbar\sim \ell^{-1}$ as $\ell\to\infty$.  The maximum has the property that\footnote{Throughout, $a\simeq 1$ means that $C^{-1}\le a\le C$ for some constant~$C$.} 
\begin{equation}
\label{eq:nondegmax}
 V'(\xphot;\hbar)=0,\quad V''(\xphot;\hbar)\simeq -1
\end{equation}
uniformly in $\hbar$ and $\xphot\simeq M$. 
 For the cosine evolution one has for $x'<x$
\begin{align}
&\cos\big(t\sqrt{\calH_{\ell,\sigma}}\big) (x,x') \nn \\&= \int_0^\infty
\cos(t\lambda)\nn
\Im \Big[ \frac{f_{+}(x,\lambda;\ell)f_{-}(x',\lambda;\ell)}{W(f_{+}(\cdot,\lambda;\ell),f_{-}(\cdot,\lambda;\ell))}
\Big] \lambda d\lambda  \\
& = \frac{2}{\pi \hbar^2} \int_0^\infty \cos(\hbar^{-1}tE) \Im \Big[
\frac{f_{+}(x,E;\hbar)f_{-}(x',E;\hbar)}{W(f_{+}(\cdot,E;\hbar),f_{-}(\cdot,E;\hbar))}
\Big] \, E\,dE    \label{eq:semievol}
\end{align}
with $f_{\pm}$ being the outgoing Jost solutions for the original operator $\calH_{\ell,\sigma}$ and
the semiclassical one, respectively.
For the latter case this means that
\begin{align*}
 (-\hbar^2 \del_x^2 + V(x;\hbar))f_{\pm}(x,E;\hbar) &= E^2 f_{\pm}(x,E;\hbar) \\
 f_{\pm}(x,E;\hbar) &\sim e^{\pm \prefac Ex} \qquad x\to \pm \infty.
\end{align*}
Furthermore, we write $W(f,g):=fg'-f'g$ for the Wronskian of two functions $f,g$. 
The sine evolution is given by
\begin{equation}\label{eq:sineevol}
 \frac{\sin\big(t\sqrt{\calH_{\ell,\sigma}}\big) (x,x')}{\sqrt{\calH_{\ell,\sigma}}}= \frac{2}{\pi \hbar}\int_0^\infty
\sin(\hbar^{-1}tE) e(E,x,x';\hbar)\, dE
\end{equation}
with the semiclassical spectral measure
\begin{equation} e(E,x,x';\hbar)=\Im \Big[
\frac{f_{+}(x,E;\hbar)f_{-}(x',E;\hbar)}{W(f_{+}(\cdot,E;\hbar),f_{-}(\cdot,E;\hbar))}
\Big]
\label{eq:esemiclass}
\end{equation}
In order to control the semiclassical evolution we distinguish energies $0<E<\eps$, $\eps\le E<100$
and $100\le E$. Here $\eps>0$ is some fixed small constant which does not depend on~$\hbar$.

The regime of large energies is relatively easy, whereas the low
lying energies as well as those near the maximum $V=1$ represent the
most difficult contributions to analyze. For small energies we
follow the analysis of~\cite{CSST} and~\cite{CST} which was specifically developed
with this application in mind. In the former paper, the challenge was to carry out
the WKB analysis for a smooth, positive, inverse square potential
{\em uniformly} for small $\hbar$ and small energies $0<E<\eps$.
This was accomplished by means of Langer's uniformizing
transformation which reduces the perturbative analysis to an Airy
equation. We note that a novel feature was the modification of the
potential, see Section~\ref{sec:low}. In this paper we have to go
beyond~\cite{CSST} since the Regge-Wheeler potential exhibits
inverse square decay only in the far field, whereas it decays
exponentially towards the event horizon. This is where~\cite{CST} applies, which 
develops a normal form reduction for the exponentially decaying region to the left
of the maximum of the potential. 

As mentioned above, we do not employ a uniformizing transformation
for energies close to the maximum $V=1$; this can indeed be done,
and requires a perturbation theory around Weber's equation\footnote{We will pursue this matter elsewhere. 
This approach seems needed in order to prove a Huyghen's principle along the lines of this paper.}.
Instead,
we prove a Mourre estimate near the maximum. This is somewhat
unusual as the maximum energy {\em is trapping} and therefore needs
to be excluded in the classical Mourre theory, see~\cite{Graf},
\cite{HisNak}. However, a simple application of the uncertainty
principle (or the ground state of the semiclassical harmonic
oscillator) allows one to deal with this trapping case as well. Of
crucial importance here is that the maximum of the potential is
nondegenerate\footnote{Technically speaking, the methods of this
paper apply to potentials with a unique maximum satisfying $V^{(j)}(\tilde x)=0$ for $1\le j\le k$,
$V^{(k+1)}(\tilde x)\ne0$. However, the number of derivatives~$\snabla$ in
our decay estimates would then increase with~$k$.} (in fact,
$V''(\xphot)<0$).
  Once the Mourre bound is established, we employ a semiclassical version of the propagation estimates of Hunziker, Sigal, Soffer~\cite{HSS}, which 
  in turn go back to the work of Sigal, Soffer~\cite{SigSof2}, see
Section~\ref{sec:HSS} below.

\subsection{Notations}
In this paper we frequently employ the notation $a \lesssim b$ (for $a,b \in \mathbb{R}$) meaning that
there exists an (absolute) constant $c>0$ such that $a\leq cb$.
We also use $a \gtrsim b$ and write $a\simeq b$ if $a \lesssim b$ and $b \lesssim a$.
Furthermore, $O(f(x))$ denotes a generic complex--valued function that satisfies $|O(f(x))|\lesssim
|f(x)|$ in a domain of $x$ that is either stated explicitly or follows from the context.
We write $O_\mathbb{R}(f(x))$ to indicate that the respective function is real--valued.
The symbol $\sim$ is reserved for asymptotic equality, i.e., $f(x) \sim g(x)$ as $x \to a$, where
$f,g$ are two complex--valued functions, means that
$$ \lim_{x \to a}\frac{f(x)}{g(x)}=1. $$

\section{Low lying energies and WKB}\label{sec:low}

In this section we bound \eqref{eq:semievol} for energies
$0<E<\eps$. Our approach is based on the perturbative WKB analysis
of the Jost solutions which was developed in~\cite{CSST} and~\cite{CST}.  More
precisely, \cite{CSST} applies to the case of $x\ge0$ for which 
the potential decays like an inverse square. For $x\le0$ the
potential exhibits exponential decay as $x\to-\infty$ and~\cite{CST} develops
the methods needed for that case. We present the main steps of the analysis developed in these papers but omit the most
involved technical details so as not to disrupt the flow of the argument.

\subsection{The far field} \label{subsec:farfield}

We begin with the former case, i.e., $x\ge0$. In fact, we  shall apply the analysis of this section to $x\ge x_0$
where $x_0<0$ is a {\em fixed} constant. In fact, any $x_0<0$ is admissible, see below.
Define a modified potential\footnote{It was shown in \cite{CSST} that
the WKB approximation can only be applied {\em after} this modification.}
\[
V_0(x;\hbar):= V(x;\hbar) + \frac{\hbar^2}{4}\la x\ra^{-2}
\]
and denote by $x_1(E;\hbar)>0$ the unique positive turning point for
any\footnote{This condition will be tacitly in force throughout this section.} $0<E<\eps$, i.e., 
the solution of $V_0(x;\hbar)=E^2$, $x>0$.

\subsubsection{Liouville-Green transform and reduction to a perturbed Airy equation}

The analysis of~\cite{CSST} was based on the following ``Langer transform'', which in turn is a special case of the Liouville-Green transform,
see~\cite{Olver}, \cite{Miller}. See Lemma~3 in \cite{CSST} for essentially the same statement.

\begin{lemma}\label{lem:langer}
With\footnote{We warn the reader that what was called $E$ in~\cite{CSST} is now called~$E^2$.} $Q_0:=V_0-E^2$
\begin{equation}
  \label{eq:zeta}
   \zeta=\zeta(x,E;\hbar):= \sign (x-x_1(E;\hbar)) \Big|\frac32 \int_{x_1(E;\hbar)}^x
\sqrt{|Q_0(u,E;\hbar)|}\, du\Big|^{\frac23}
\end{equation}
defines a smooth change of variables $x\mapsto\zeta$
for all $x\ge x_0$. Let $q:=-\frac{Q_0}{\zeta}$. Then  $q>0$, $\frac{d\zeta}{dx}=\zeta'=\sqrt{q}$, and
\[
-\hbar^2f''+(V-E^2)f= 0
\]
transforms into
\begin{equation}\label{eq:Airy}
-\hbar^2 \ddot w(\zeta) = (\zeta+\hbar^2 \tilde V(\zeta,E;\hbar))w(\zeta)
\end{equation}
under $w= \sqrt{\zeta'} f = q^{\frac14} f$. Here $\dot{\
}=\frac{d}{d\zeta}$ and
\[
\tilde V := \frac{1}{4} q^{-1} \langle x\rangle ^{-2} - q^{-\frac14} \frac{d^2
q^{\frac14}}{d\zeta^2}
\]
\end{lemma}
\begin{proof}
 It is clear that~\eqref{eq:zeta} defines a smooth map away from the point $x=x_1(E;\hbar)$.
Taylor expanding $Q_0(x,E;\hbar)$ in a neighborhood of that point
and using that $V_0'(x_1(E;\hbar))<0$ implies that
$\zeta(x,E;\hbar)$ is smooth around $x=x_1$ as well with
$\zeta'(x_1,E;\hbar)>0$.  Next, one checks that
\begin{align*}
 \dot w &= q^{-\frac14} f' + \frac{d q^{\frac14}}{d\zeta} f,\quad
 \ddot w = q^{-\frac34} f'' + \frac{d^2 q^{\frac14}}{d\zeta^2} f
\end{align*}
and thus, using $-\hbar^2f'' = (E^2-V)f$,
\begin{align*}
 -\hbar^2 \ddot w &= q^{-1} (E^2-V) w - \hbar^2 q^{-\frac14} \frac{d^2 q^{\frac14}}{d\zeta^2} w \\
&=  q^{-1} (-Q_0 + \hbar^2 \la x\ra^{-2} /4) w - \hbar^2 q^{-\frac14} \frac{d^2 q^{\frac14}}{d\zeta^2} w \\
&= \zeta w(\zeta) + \hbar^2 \big(q^{-1}\la x\ra^{-2}/4 - q^{-\frac14} \frac{d^2 q^{\frac14}}{d\zeta^2}\big) w
\end{align*}
as claimed.
\end{proof}

The Airy equation~\eqref{eq:Airy} provides a convenient way of
solving the matching problem\footnote{The usual WKB machinery
requires solving two matching problems, namely between the Airy
region and the oscillatory region on the one hand, and the Airy
region and the exponential growth/decay region on the other hand;
see for example~\cite{Miller}.}
 at the turning point~$\zeta=0$.  We remark that it was assumed in~\cite{CSST}
that the potential satisfies $V(x)=\mu x^{-2}+O(x^{-3})$ as $x\to \infty$.  However, the methods of that paper
equally well apply under the weaker assumption $V(x)=\mu x^{-2}+O(x^{-\sigma})$ as $x\to \infty$ where $\sigma>2$; furthermore, one needs
$\del_x^k O(x^{-\sigma}) = O(x^{-\sigma-k})$ for all $k\ge0$. This is relevant here since the Regge-Wheeler potential exhibits
a $\frac{\log x}{x^3}$-correction to the leading $x^{-2}$-decay.

\subsubsection{A basis for the perturbed Airy equation} For the Airy functions $\Ai, \Bi$ appearing below we
refer the reader to Chapter~11 of~\cite{Olver}. To the left of the
turning point a fundamental system of~\eqref{eq:Airy} is described
by the following result, see Proposition~8 in~\cite{CSST}.

\begin{prop}
  \label{prop:AiryI} Let $\hbar_0>0$ be small. A fundamental system of solutions to~\eqref{eq:Airy} in the range
$\zeta\le0$ is given by
\begin{equation}\nonumber
\begin{aligned}
  \phi_1(\zeta,E,\hbar) &=
  \Ai(\tau) [1+\hbar a_1(\zeta,E,\hbar)] \\
  \phi_2(\zeta,E,\hbar) &=
  \Bi(\tau) [1+\hbar a_2(\zeta,E,\hbar)]
\end{aligned}
\end{equation}
 with
$\tau:=-\hbar^{-\frac23}\zeta$.
 Here $a_1, a_2$ are smooth, real-valued, and they satisfy
the bounds, for all $k\ge0$ and $j=1,2$, and with $\zeta_0:=\zeta(x_0,E)$,
\begin{equation}\label{eq:aj_est}
\begin{aligned}
  \sup_{\zeta_0\le\zeta\le0}|\partial_E^k  a_j(\zeta,E,\hbar)| &\less  E^{-k}  \\
 |\partial_E^k \partial_\zeta a_j(\zeta,E,\hbar)| &\les E^{-k} \Big[\hbar^{-\frac13} \la
\hbar^{-\frac23}\zeta\ra^{-\frac12}\chi_{[-1\le \zeta\le0]} +
|\zeta|^{\frac12}  \chi_{[ \zeta_0\le \zeta\le -1]}\Big ]
\end{aligned}
\end{equation}
uniformly in the parameters $0<\hbar<\hbar_0$, $0<E<\eps$.
\end{prop}
\begin{proof}
 This is essentially Proposition~8 in~\cite{CSST}. The two differences are (i) we work with $E^2$ instead
of~$E$ (ii) the potential has a $\frac{\log x}{x^3}$-correction to the leading inverse square decay rather than
the $x^{-3}$-correction assumed in~\cite{CSST}.

As far as (i) is concerned, we note the following. Let $E$ be as in~\cite{CSST} and assume $\tilde E^2:=E$.
If $|\del_E a(E)|\les E^{-1}$,
then $b(\tilde E):= a(E)$ satisfies $|\del_{\tilde E} b(\tilde E)|\les \tilde E^{-1}$ by the chain rule.
So it makes no difference whether we work with $E$ or~$\tilde E$.
As for (ii), we note that the only change to the estimates in Section~3 of~\cite{CSST} is in Lemma~7, where one needs
to replace $\la x\ra^{-3}\beta_1(x,E)$ with $\la x\ra^{-3}\log x\, \beta_1(x,E)$, cf.~(3.20) and~(3.24) in that paper.
However, inspection of the proof of Proposition~8 there reveals that this does not affect the resulting bound in any way,
see the paragraph between~(4.12) and~(4.13) there.
\end{proof}

We remark that Proposition~\ref{prop:AiryI} would fail if we had defined $\zeta$ in~\eqref{eq:zeta} with $Q=V-E^2$ instead of~$Q_0$.
In the region $\zeta\ge0$ we have a basis of oscillatory solutions as described by the following result, see Proposition~9 in~\cite{CSST}.

\begin{prop}
  \label{prop:AiryII} Let $\hbar_0>0$ be small. In the range
$\zeta\ge 0$ a basis of solutions to~\eqref{eq:Airy} is given by
\begin{equation}\nonumber
\begin{aligned}
  \psi_1(\zeta,E;\hbar) &=
  (\Ai(\tau)+i\Bi(\tau)) [1+\hbar b_1(\zeta,E;\hbar)] \\
  \psi_2(\zeta,E;\hbar) &=
  (\Ai(\tau)-i\Bi(\tau)) [1+\hbar b_2(\zeta,E;\hbar)]
\end{aligned}
\end{equation}
 with
$\tau:=-\hbar^{-\frac23}\zeta$ and where
$b_1, b_2$  are smooth, complex-valued, and satisfy the bounds
for all $k\ge0$, and $j=1,2$
\begin{equation}\label{eq:bj_est}
\begin{aligned}
 | \partial_E^k\, b_j(\zeta,E;\hbar)| &\le C_{k}\, E^{-k} \la\zeta\ra^{-\frac32} \\
|\partial_\zeta\partial_E^k b_j(\zeta,E;\hbar)| &\le C_{k}\, E^{-k} \hbar^{-\frac13} \la\zeta\ra^{-2}
\end{aligned}
\end{equation}
uniformly in the parameters $0<\hbar<\hbar_0$, $0<E<\eps$,
$\zeta\ge0$.
\end{prop}

\subsubsection{The outgoing Jost solution for the far field}\label{subsubsec:jostright}

We can now draw the following conclusions from
Propositions~\ref{prop:AiryI} and~\ref{prop:AiryII} about the
outgoing Jost solution $f_{+}$. First, recall that the Airy functions
satisfy the following asymptotic expansions with
$\xi=\frac23x^{\frac32}$:
\begin{align}
\nn \Ai(x)&= \frac{e^{-\xi}}{2\pi^{\frac12}x^{\frac14}}(1+O(\xi^{-1})),\quad \Bi(x)=\frac{e^{\xi}}{\pi^{\frac12}x^{\frac14}}(1+O(\xi^{-1}))   \\
\label{eq:Airyasymp} \Ai(-x)&= \frac{1}{\pi^{\frac12}x^{\frac14}}\big[\cos(\xi-\frac14\pi)(1+O(\xi^{-2}))+\sin(\xi-\frac14\pi)O(\xi^{-1})\big],\\
 \nn   \Bi(-x)&=\frac{1}{\pi^{\frac12}x^{\frac14}}\big[-\sin(\xi-\frac14\pi)(1+O(\xi^{-2})) + \cos(\xi-\frac14\pi)O(\xi^{-1})\big]
\end{align}
as $x\to\infty$, see~\cite{Olver}. Moreover,
\[
 W(\Ai,\Bi)=\Ai\,\Bi'-\Ai'\,\Bi=\frac{1}{\pi}
\]
In what follows,
\begin{align*}
 S_+(E;\hbar) &:=  \int_{x_0}^{x_1(E;\hbar)} \sqrt{V_0(y;\hbar)-E^2}\, dy  \\
T_+(E;\hbar) &:= Ex_1(E;\hbar) - \int_{x_1(E;\hbar)}^\infty
\big(\sqrt{V_0(y;\hbar)-E^2}-E \big)\, dy
\end{align*}
One checks that $S_+(E;\hbar)\sim |\log E|$ as $E\to0+$, whereas $T_+(E;\hbar)\to T_+(0;\hbar)$, some finite number. Moreover,
\[
 |\del_E^k S(E;\hbar)|+| \del_E^k T_+(E;\hbar)|\le C_k\, E^{-k}
\]
for all $k\ge1$.
One has\footnote{We suppress $\hbar$ as argument in most functions, even though everything here does depend on~$\hbar$.}
\begin{equation}\label{eq:f+psi2}
 f_+(x,E ) = \sqrt{\pi}\,E^{\frac12}\hbar^{-\frac16}e^{i(\frac{T_+(E)}{\hbar}+\frac{\pi}{4})} q^{-\frac14}(\zeta) 
 \psi_2(\zeta,E ).
\end{equation}
This is obtained by matching the asymptotic behavior of $f_+$ with that of $\psi_2(\zeta)$ as $x\to\infty$
and by using the relation $w=q^{\frac14}f$ from Lemma~\ref{lem:langer}.
We refer the reader to \cite{CSST} for all the details.
We now connect $\psi_2$ to the basis $\phi_j(\zeta,E )$ of Proposition~\ref{prop:AiryI}:
\[
 \psi_2(\zeta,E ) = c_1(E ) \phi_1(\zeta,E ) + c_2(E ) \phi_2(\zeta,E )
\]
where
\[
 c_1(E ) = \frac{W(\psi_2(\cdot,E ),\phi_2(\cdot,E ))}{W(\phi_1(\cdot,E ), \phi_2(\cdot,E ))}, \quad c_2(E ) = - \frac{W(\psi_2(\cdot,E ),
 \phi_1(\cdot,E ))}{W(\phi_1(\cdot,E ), \phi_2(\cdot,E ))}
\]
By Proposition~\ref{prop:AiryI},
\[
 W(\phi_1(\cdot,E ), \phi_2(\cdot,E )) = -\hbar^{-\frac23} W(\Ai,\Bi) + O(\hbar^{\frac13}) =-\pi^{-1}\hbar^{-\frac23}(1+O(\hbar))
\]
where we evaluated the Wronskian on the left-hand side at $\zeta=0$. Next, by Propositions~\ref{prop:AiryI} and~\ref{prop:AiryII},
\begin{equation}
 \label{eq:Wphipsi}
\begin{aligned}
  W(\psi_2(\cdot,E ),\phi_2(\cdot,E )) &= -\hbar^{-\frac23}[(\Ai(0)-i\Bi(0))\Bi'(0)\\
&\qquad -(\Ai'(0)-i\Bi'(0))\Bi(0)+O(\hbar)] \\
&= -\hbar^{-\frac23}[W(\Ai,\Bi) +O(\hbar)] \\
W(\psi_2(\cdot,E ),\phi_1(\cdot,E )) &= -\hbar^{-\frac23}[(\Ai(0)-i\Bi(0))\Ai'(0)\\
&\qquad -(\Ai'(0)-i\Bi'(0))\Ai(0)+O(\hbar)] \\
&= -\hbar^{-\frac23}[iW(\Ai,\Bi) +O(\hbar)]
\end{aligned}
\end{equation}
so that
\begin{equation}\label{eq:c1c2}
 c_1(E )=1+O(\hbar),\quad c_2(E )=-i+O(\hbar)
\end{equation}
where the $O(\cdot)$ terms satisfy $|\partial_E^k O(\hbar)|\le C_k\,\hbar E^{-k}\ $.
With $\zeta_0=\zeta(x_0,E)$ we infer that
\begin{align*}
 f_+(x_0,E ) &= \sqrt{\pi}e^{i(\frac{T_+(E)}{\hbar}+\frac{\pi}{4})}\,E^{\frac12}\hbar^{-\frac16} q^{-\frac14}(\zeta_0)\psi_2(\zeta_0,E ) \\
&= \sqrt{\pi}e^{i(\frac{T_+(E)}{\hbar}+\frac{\pi}{4})}\,E^{\frac14}\hbar^{-\frac16} q^{-\frac14}(\zeta_0) [ c_1(E )\phi_1(\zeta_0,E ) + c_2(E ) \phi_2(\zeta_0,E ) ]\\
 f_+'(x_0,E ) &= \sqrt{\pi}e^{i(\frac{T_+(E)}{\hbar}+\frac{\pi}{4})}\,E^{\frac12}\hbar^{-\frac16}\zeta'(x_0)q^{-\frac14}(\zeta_0)\big[\dot\psi_2(\zeta_0,E )-\frac14 \frac{\dot q}{q}(\zeta_0)\psi_2(\zeta_0,E )  \big] \\
&=\sqrt{\pi}e^{i(\frac{T_+(E)}{\hbar}+\frac{\pi}{4})}\,E^{\frac12}\hbar^{-\frac16}q^{\frac14}(\zeta_0)\big[c_1(E )\dot\phi_1(\zeta_0,E ) + c_2(E ) \dot\phi_2(\zeta_0,E ) \\
&\quad -\frac14 \frac{\dot q}{q}(\zeta_0)(c_1(E )\phi_1(\zeta_0,E ) + c_2(E ) \phi_2(\zeta_0,E ) )  \big]
\end{align*}
where we have used that $\zeta'=q^{\frac12}$, see Lemma~\ref{lem:langer}. Furthermore,
\[
 \frac{\dot q}{q}(\zeta_0) = -\zeta_0^{-1} + |\zeta_0|^{\frac12} \frac{V_0'(x_0)}{(V_0(x_0)-E^2)^{\frac32}} = O(|\zeta_0|^{\frac12})
\]
with $O(\del_E^k |\zeta_0|^{\frac12}) = O(E^{-k})$ for $k\ge1$.
From Proposition~\ref{prop:AiryI}, with $O_\R$ denoting a real-valued term, 
\begin{align*}
\phi_1(\zeta_0,E ) &= \Ai(-\hbar^{-\frac23}\zeta_0)(1+O_\R(\hbar))\\
 \phi_2(\zeta_0,E ) &=  \Bi(-\hbar^{-\frac23}\zeta_0)(1+O_\R(\hbar))\\
 \dot\phi_1(\zeta_0,E ) &= -\hbar^{-\frac23}\Ai'(-\hbar^{-\frac23}\zeta_0)(1+O_\R(\hbar))+O_\R(\hbar)|\zeta_0|^{\frac12}\Ai(-\hbar^{-\frac23}\zeta_0)\\
 \dot\phi_2(\zeta_0,E ) &= -\hbar^{-\frac23}\Bi'(-\hbar^{-\frac23}\zeta_0)(1+O_\R(\hbar))+O_\R(\hbar)|\zeta_0|^{\frac12}\Bi(-\hbar^{-\frac23}\zeta_0)
\end{align*}
which implies via \eqref{eq:Airyasymp} that
\begin{align*}
\phi_1(\zeta_0,E ) &= (4\pi)^{-\frac12}(\hbar^{-\frac23}|\zeta_0|)^{-\frac14}e^{-\frac23 \hbar^{-1}|\zeta_0|^{\frac32}}(1+O_\R(\hbar))\\
 \phi_2(\zeta_0,E ) &= \pi^{-\frac12}(\hbar^{-\frac23}|\zeta_0|)^{-\frac14}e^{\frac23 \hbar^{-1}|\zeta_0|^{\frac32}}(1+O_\R(\hbar))\\
 \dot\phi_1(\zeta_0,E ) &= \hbar^{-\frac23}(4\pi)^{-\frac12}(\hbar^{-\frac23}|\zeta_0|)^{\frac14}e^{-\frac23 \hbar^{-1}|\zeta_0|^{\frac32}}(1+O_\R(\hbar))\\
 \dot\phi_2(\zeta_0,E ) &= -\hbar^{-\frac23}\pi^{-\frac12}(\hbar^{-\frac23}|\zeta_0|)^{\frac14}e^{\frac23 \hbar^{-1}|\zeta_0|^{\frac32}} (1+O_\R(\hbar))
\end{align*}
In view of these properties and using that $e^{-\hbar^{-1}|\zeta_0|^{\frac32}} = O_\R(\hbar)$ where $\partial_E^k O_\R(\hbar) = O(E^{-k}\hbar)$,
one obtains (with $c_2$ as above)
\begin{equation}
 \begin{aligned}
  f_+(x_0,E ) &= c_2\, \gamma\,  q^{-\frac14}(\zeta_0) \phi_2(\zeta_0,E )  \big[ (1+O_{\R}(\hbar)) +i(\frac12+O(\hbar)) e^{-2\hbar^{-1}S_+} \big ] \\
f_+'(x_0,E ) &=  c_2\, \gamma\, q^{\frac14}(\zeta_0) \dot\phi_2(\zeta_0,E )  \big[ (1+O_{\R}(\hbar)) -i(\frac12+O(\hbar)) e^{-2\hbar^{-1}S_+} \big ]
 \end{aligned}
\label{eq:f_+andder}
\end{equation}
where $\gamma=\gamma(E,\hbar):= -\sqrt{\pi}\,e^{i(\frac{T_+(E)}{\hbar}+\frac{\pi}{4})}\,E^{\frac12}\hbar^{-\frac16}$ and  with
\[
\frac23\,|\zeta_0|^{\frac32}= S_+(E;\hbar) =S_+ = \int_{x_0}^{x_1(E;\hbar)} \sqrt{V_0(x;\hbar)-E^2}\,dx
\]
being the action integral defined earlier. 
Furthermore, it follows from Propositions~\ref{prop:AiryI} and~\ref{prop:AiryII} that each differentiation in~$E$ loses one power of~$E$ (in particular,
the $O(\hbar)$ terms have this property). 
For future reference, we remark that
\begin{equation}\label{eq:f+quot}
\frac{f_+'(x_0,E )}{f_+(x_0,E )}=-d_1\hbar^{-1}(1+O_\R(\hbar)) \big [1+ O(e^{-2\hbar^{-1}S_+}) \big ]
\end{equation}
where $d_1>0$ is a constant (depending on~$x_0$). In particular, 
\[
\Im\Big[\frac{f_+'(x_0,E )}{f_+(x_0,E )} \Big] = \hbar^{-1}  O(e^{-2\hbar^{-1}S_+(E;\hbar)})
\]

\subsection{Approaching the event horizon}\label{subsec:event}

We now deal with the potential for $x\le  x_0$. Here $x_0<0$ is chosen such that the Regge-Wheeler potential 
(setting $2M=1$ for simplicity) can be written as
\[
 V(x;\hbar)=\sum_{n=1}^\infty c_{n-1}(\hbar) e^{nx}
\]
as a convergent series $x\le x_0$ uniformly in
$\hbar\in(0,\hbar_0]$.  In fact, since the Lambert function $W(z)$ defined via  $W(z) e^{W(z)}=z$ is analytic on $|z|<e^{-1}$, it follows that one can take any $x_0<0$. 
The coefficients have expansions in powers
of~$\hbar$ and we normalize such that $c_0(\hbar)=1+O(\hbar^2)$.   
One can also check that $c_1(0)\ne0$.  The goal is to control the Jost solutions
$f_-(x,E;\hbar)$ as $x\to-\infty$ uniformly for $(E,\hbar)\in
(0,\eps)\times (0,\hbar_0)$.

\subsubsection{Transforming the problem to a compact interval}

For notational convenience, we switch from $x$ to $-x$ and consider
$x>|x_0|$. The problem is then to control $f_+$ for the problem
\begin{equation}
\label{eq:fV}
 -\hbar^2 f_+''(x,E;\hbar) + V(x;\hbar) f_+(x,E;\hbar) = E^2 f_+(x,E;\hbar)
\end{equation}
with
\[
V(x;\hbar)=e^{-x} (1+ \sum_{n=1}^\infty c_n(\hbar) e^{-nx}) 
\]
We now transform this case into a semi-classical scattering problem
on a bounded interval $(0,y_0)$ by introducing the new independent
variable $y=2e^{-\frac{x}{2}}$. 
Setting  $f(x)=g(y)$ reduces finding the outgoing Jost solution to the equation for $g(y)=g(y,E;\hbar)$, 
\[
-\hbar^2 [g''(y)+y^{-1} g'(y)] + \Big(\Omega(y;\hbar)
-\frac{4E^2}{y^2}\Big)g(y)=0
\]
with the normalization $g(y)\sim (y/2)^{-2i\frac{E}{\hbar}}$ as
$y\to0+$, and with
\begin{equation}
\Omega(y;\hbar) = 1+ \sum_{n=1}^\infty \frac{c_n(\hbar)}{4^n} y^{2n}
\label{eq:phiseries}
\end{equation}
analytic in $|y|<y_0:=2e^{-x_0/2}$.
Finally, setting $\tilde g(y;\hbar):=y^{\frac12}g(y;\hbar)$ yields the
equation
\begin{equation}
  \label{eq:tildeggleichung}
  -\hbar^2 \tilde g''(y,E;\hbar) +\Big(\Omega(y;\hbar)-\big(\frac{\hbar^2}{4}+4E^2\big)y^{-2}\Big)
  \tilde g(y,E;\hbar) =0
\end{equation}
with the normalization
\begin{equation}
 \label{eq:tildegasymp}
\tilde g(y,E;\hbar)\sim
2^{2i\frac{E}{\hbar}}\, y^{\frac12-2i\frac{E}{\hbar}}
\end{equation}
 as $y\to0+$.
We remark that in the case $\Omega(y;\hbar)\equiv 1$ the
equation~\eqref{eq:tildeggleichung} is a modified Bessel equation
with a basis given by the modified Bessel functions
$I_{i\nu}(\hbar^{-1}y)$ and $K_{i\nu}(\hbar^{-1}y)$ with
$\nu=2\frac{E}{\hbar}$. It is shown in~\cite{CST} by means of a suitable Liouville-Green
transform  that this basis leads 
to an actual basis of~\eqref{eq:tildeggleichung}.  We begin with the following normal form result from~\cite{CST}
which is based on a Liouville-Green transform (the variable $z$ below is a rescaling of $y$: $y=\alpha z$, with $\alpha:=\sqrt{\hbar^2/4+4E^2}$). 
Recall that $f\simeq1$ means that $C^{-1}<f<C$ for some constant~$C$. 

\begin{lemma}[\cite{CST}]
\label{lem:normform}
Let $\Omega$ be as above and $\alpha_0>0$ be sufficiently small.   For all
$0<\alpha<\alpha_0$ 
 there exists a $C^\infty$ diffeomorphism  $w=w(z,\alpha): I_0(\alpha):=(0,\alpha^{-1}y_0)\to J_0(\alpha):=(0,\alpha^{-1}w_0(\alpha))$ 
 where $y_0$ is as above with the following properties, uniformly in~$0<\alpha<\alpha_0$ :
\begin{itemize}
\item $w_0(\alpha)\simeq 1$ 
\item $w'(z,\alpha)\simeq 1$  for all $z\in I_0(\alpha)$
\item $|\del_\alpha^k \del_z^\ell w(z,\alpha)|  \le C_{k,\ell} \, \la z\ra^{1+k}\, \alpha^{\ell}$ for all $k,\ell\ge0$ and $z\in I_0(\alpha)$
\end{itemize} 
Let $\hbar_1:=\alpha^{-1} \hbar$. Then there exists a function $V_2(\cdot,\alpha;\hbar)$ such that
$\psi$ solves the rescaled form of~\eqref{eq:tildeggleichung}, viz.
\begin{equation}\label{eq:psiz}
-\hbar_1^2 \psi''(z)+ \big(\Omega(\alpha z;\hbar)-z^{-2}\big)\psi(z)=0
\end{equation}
on $I_0(\alpha)$ iff $\vphi(w):= (w'(z,\alpha))^{\frac12} \psi(z)$ (where $w=w(z,\alpha)$) solves
\begin{equation}\label{eq:wgleichung}
-\hbar_1^2 \vphi''(w)+ \big(1-w^{-2}\big)\vphi(w)= \hbar_1^2 V_2(w) \vphi(w)
\end{equation}
on $J_0(\alpha)$. Furthermore, the potential $V_2(\cdot,\alpha;\hbar)$ satisfies 
\[
|\del_\alpha^k \del_w^\ell  V_2(w,\alpha;\hbar)|\le C_{k,\ell}\, \la w\ra^{1+k} \alpha^{3+\ell}
\]
for all $w\in J_0(\alpha)$ and $k,\ell\ge0$. 
\end{lemma}
\begin{proof}
This is done by setting
\[
\frac{dw}{dz}:= \sqrt{\frac{z^{-2}-\Omega(\alpha z;\hbar)}{w^{-2}-1}}
\]
More precisely, with $z_t$ being the turning point defined by $z_t^{-2}-\Omega(\alpha z_t;\hbar)=0$, this means that 
\[
\int_{1}^{w} \sqrt{1-v^{-2}}\, dv = \int_{z_{t}}^{z} \sqrt{\Omega(\alpha u;\hbar)-u^{-2}}\,du
\]
provided $z>z_t$ and 
\[
\int^{1}_{w} \sqrt{-1+v^{-2}}\, dv = \int^{z_{t}}_{z} \sqrt{-\Omega(\alpha u;\hbar)+u^{-2}}\,du
\]
provided $0<z<z_{t}$. Note that $w\to0$ as $z\to0$.   The properties of $w$ stated above are now
shown by calculus. 
 The potential $V_2$ is given by
\[
V_2(w)= (w'(z))^{-\frac32}  \del_z^2  (w'(z))^{-\frac12} = \frac34 \frac{(w''(z))^2}{(w'(z))^4} - \frac12 \frac{w'''(z)}{(w'(z))^3}.
\]
We refer the reader to~\cite{CST} for further details. 
\end{proof}

We remark that the proof also shows that $w(z)=z+O(z^2)$ and $w'(z)=1+O(z)$ as $z\to0$. 
One now concludes the following concerning a basis of~\eqref{eq:tildeggleichung}. 
Let $\alpha:=\sqrt{\hbar^2/4+4E^2}$. 
Since \eqref{eq:psiz} is a rescaled form of~\eqref{eq:tildeggleichung}, one can now obtain a system of
fundamental solutions to the latter equation from a perturbative analysis of~\eqref{eq:wgleichung}. 
The modified Bessel functions $I_{i\nu}(z)$ and $K_{i\nu}(z)$, which are both analytic on $\C\setminus(-\infty,0]$, give rise to a fundamental
system of the homogeneous equation on the left-hand side of~\eqref{eq:wgleichung}. In our case $\nu=2\frac{E}{\hbar}$.  Recall the asymptotics
\begin{equation}\label{eq:Jnu}
I_{i\nu}(z)= \frac{(z/2)^{i\nu}}{\Gamma(i\nu+1)}(1 + O(z^2)) \quad z\to0.
\end{equation}
Note that for our purposes it suffices to consider real $z$.
Moreover, $I_{i\nu}(x)$ grows exponentially as $x\to\infty$,
whereas $K_{i\nu}(x)$ decays exponentially as $x\to\infty$. 

\begin{cor}
\label{cor:tildebasis}  Let $\alpha:=\sqrt{\hbar^2/4+4E^2}$ with $\hbar$ and $E>0$ small.
There exists a  fundamental system  of~\eqref{eq:tildeggleichung}, denoted by  $(\tilde g_{0},\tilde g_{1})$,  of the form 
\begin{align*}
\tilde g_{0}(y,E;\hbar) &= (w(z)/w'(z))^{\frac12} I_{2i\frac{E}{\hbar}} (\frac{w(z)}{\hbar_1}) (1+\hbar c_{1}(y,E;\hbar))   \\
\tilde g_{1}(y,E;\hbar) &= (w(z)/w'(z))^{\frac12} I_{-2i\frac{E}{\hbar}} (\frac{w(z)}{\hbar_1}) (1+\hbar c_{2}(y,E;\hbar))   
\end{align*}
where $w(z)=w(z,E;\hbar)$ is as in Lemma~\ref{lem:normform}, and with 
$z=\frac{y}{\alpha}$, $\hbar_1=\frac{\hbar}{\alpha}$. The 
$c_{j}$ satisfy  for all $k,\ell\ge0$, 
\[
|\del_{E}^{k} \del_{y}^{\ell} c_{j}(y,E;\hbar)|\le C_{k,\ell} \alpha^{-k}
\]
and all $0<y<y_0$. 
\end{cor}
\begin{proof}
This follows from two  facts: (i) a basis of the homogeneous equation~\eqref{eq:wgleichung} is given by
\[
\phi_0(w,E;\hbar):= \sqrt{w}I_{2i\frac{E}{\hbar}}(\frac{w}{\hbar_1}),\quad \phi_1(w,E;\hbar):=\sqrt{w} 
I_{-2\frac{E}{\hbar}} (i\frac{w}{\hbar_1})
\]
and (ii): the equations for $c_{1,2}$ are contractive; in fact, they are given by the usual Volterra equation involving
the homogeneous basis and the potential $V_{2}$.    For $c_1$ one has (suppressing $E$ and $\hbar$ as arguments)
\[
c_1(w)=-\hbar_1^{-1}\int_0^w \int_u^w \phi_0^{-2}(v)\, dv \, V_2(u)\phi_0^2(u)(1+\hbar_1 c_1(u))\, du
\]
which implies the desired bounds on~$c_1$ via Lemma~\ref{lem:normform} and the well-known asymptotic
behavior of the modified Bessel   functions. For this see~\cite{CST}. 
\end{proof}

\subsubsection{The outgoing Jost solution towards the event horizon}
\label{subsubsec:jostleft}

From \eqref{eq:Jnu}, Lemma~\ref{lem:normform}, and Corollary~\ref{cor:tildebasis} we conclude that 
\[
\tilde g_1(y,E;\hbar) = \Big(\frac{y}{\alpha}\Big)^{\frac12} \frac{(iy/2\hbar)^{-2i\frac{E}{\hbar}}}{\Gamma(1-2i\frac{E}{\hbar})}+o(1)
\]
as $y\to0$.  In view of \eqref{eq:tildegasymp} this implies that the outgoing Jost solution is represented as
\begin{align*}
f_-(x,E;\hbar)&=\sqrt{\alpha}\frac{\Gamma(1-i\nu)}{(-i\hbar)^{i\nu}}\tilde g_1(y,E;\hbar)\\
&=
\frac{\Gamma(1-i\nu)}{(-i\hbar)^{i\nu}} (\alpha w(z)/w'(z))^{\frac12} I_{-i\nu} (\frac{\alpha w(z)}{\hbar}) (1+O(\hbar))
\end{align*}
for all $x\le x_0$ with $y=\alpha z=2e^{\frac{x}{2}}$ and $\nu=2\frac{E}{\hbar}$.  
In particular, using the standard asymptotic behavior of $I_{-i\nu}$, see~\cite{CST}, 
one obtains for $x=x_0$ that 
\begin{equation}\label{eq:f_-andder}
\begin{aligned}
f_-(x_0,E;\hbar)&=  \gamma_- \sqrt{\alpha} \, e^{\frac{1}{\hbar}(S_{-}(E;\hbar)+iT_{-}(E;\hbar) )} (1+O(e^{-\frac{2}{\hbar}S_-}))(1+O_\R(\hbar))\\
f_-'(x_0,E;\hbar)&=  \gamma_-' \sqrt{\alpha} \, \hbar^{-1}  e^{\frac{1}{\hbar}(S_{-}(E;\hbar)+iT_{-}(E;\hbar))}(1+O(e^{-\frac{2}{\hbar}S_-})) (1+O_\R(\hbar))
\end{aligned}
\end{equation}
with constants $|\gamma_-|\simeq |\gamma_-'|\simeq1$ depending on~$E,\hbar$, as well as a suitable action $S_-(E;\hbar)$ which is analytic for $|E|\les\hbar$
with $S_-(E;\hbar)>0$ for small real-valued $E$, and $T_-(E;\hbar)$ some  real-valued function of real~$E$ analytic on $|E|\les\hbar$. We remark that
$\frac{\gamma_-}{\gamma_-'}>0$, which is most important in Section~\ref{subsec:e00}.  Furthermore, each derivative in~$E$ costs at most a  power of~$\hbar^{-1}$. It is important
that one does not lose~$E^{-1}$ as in the $x\ge0$ case, but only $ O(\hbar^{-1})$ as such a loss is negligible compared to the size of~$ e^{\frac{1}{\hbar}S_{-}(E;\hbar)}$. 

\subsection{The Wronskian of the outgoing Jost solutions}

From Sections~\ref{subsubsec:jostright} and \ref{subsubsec:jostleft} it is now a simple matter to
determine the Wronskian between the outgoing Jost solutions.

\begin{lemma}
 \label{lem:Wronski}
Define
\[
 S(E;\hbar):=S_+(E;\hbar)+S_-(E;\hbar), \quad T(E;\hbar):=T_+(E;\hbar)+T_-(E;\hbar)
\]
 One has
\begin{equation}
 \label{eq:Wronski}
 W(f_+(\cdot,E;\hbar),f_-(\cdot,E;\hbar)) = \gamma_0 \, E\hbar^{-1} e^{\hbar^{-1}(S(E;\hbar)+iT(E;\hbar))}(1+O(\hbar))
\end{equation}
where $\gamma_0\ne0$ is an absolute constant,  and $|\del_E^k O(\hbar)| \le C_k\,\hbar E^{-k}$
for all $0<E\ll 1$, $0<\hbar\ll 1$ and $k\ge1$.
\end{lemma}
\begin{proof}
This follows from~\eqref{eq:f_+andder} and \eqref{eq:f_-andder}.
\end{proof}

Due to the growth of the action $S(E;\hbar)$ one can now conclude the following important size estimate on the Wronskian:
\begin{equation}
 \nn
|W(f_+(\cdot,E;\hbar),f_-(\cdot,E;\hbar)) | \simeq \hbar^{-1} E e^{\hbar^{-1}S(E;\hbar)}\gtrsim N  (\mu E)^{1-N}, \qquad N:=\hbar^{-1}
\end{equation}
for all $0<E<\eps$. More precisely, one uses that $S_+(E;\hbar)=-\log E+\alpha_0 + o(1)$ as $E\to 0+$ uniformly in small~$\hbar$,
whereas $S_-(E;\hbar)>0$ for small $E$. Then $\mu:= e^{-\alpha_0}$. In other words, the Wronskian blows up as $E\to0+$ as a power law
with large power since $\hbar$ is very small. 

\subsection{The spectral measure $e(E,x,x';\hbar)$ near the maximum of the potential}
\label{subsec:e00}

We now derive the contribution of energies $0<E<\eps$ to the
 desired pointwise decay of \eqref{eq:semievol} in time. We shall fix
  $x=x'=x_{0}$ since this case can be treated most easily from the previous sections; moreover, the region
  near the maximum of the potential is in some sense the most important one.
  The case of general $x,x'$ is considered in Section~\ref{subsec:eL2}.
First, one has
\begin{align*}
e(E;x_{0},x_{0};\hbar)&=\Im\Big[\frac{f_{-}(x_{0},E;\hbar)f_{+}(x_{0},E;\hbar)}{f_{+}(x_{0},E;\hbar)f_{-}'(x_{0},E;\hbar)-
f_{+}'(x_{0},E;\hbar)f_{-}(x_{0},E;\hbar)}  \Big]\\
&=   \Im\Big[\frac{f_{-}'(x_{0},E;\hbar)}{f_{-}(x_{0},E;\hbar)}-
\frac{f_{+}'(x_{0},E;\hbar)}{f_{+}(x_{0},E;\hbar)}  \Big]^{-1}\\
&= \hbar\Im\Big[\frac{\alpha(E;\hbar)}{1-\alpha(E;\hbar)\beta(E;\hbar)}\Big]
\end{align*}
where
\[
\alpha(E;\hbar):=\hbar^{-1}\frac{f_{-}(x_{0},E;\hbar)}{f_{-}'(x_{0},E;\hbar)},\qquad \beta(E;\hbar):=\hbar\frac{f_{+}'(x_{0},E;\hbar)}{f_{+}(x_{0},E;\hbar)}
\]
From \eqref{eq:f_-andder} one has\footnote{This $\alpha$ is not related to $\alpha$ appearing in Section~\ref{subsec:event}.}
\[
\alpha(E;\hbar):= d_{0}\,  [1+ O(e^{-2\hbar^{-1} S_-}  ) ]   (1+O_\R(\hbar))
\]
where $d_{0}>0$ is a  constant that depends on $x_{0},E,\hbar$ with $d_0\simeq1$. The $O(\hbar)$-terms in the numerator and denominator are not necessarily the same.
Similarly, from~\eqref{eq:f+quot}, with a constant $d_1>0$, 
\[
\beta(E;\hbar):= -d_1[1+O(e^{-2\hbar^{-1} S_+} )]   (1+O_\R(\hbar)) = -\tilde d_1(\hbar) [1+O(e^{-2\hbar^{-1} S_+} )]
\]
Due to the exponential decay of $V$ as $x\to-\infty$, the functions $f_-(x,E;\hbar)$ and $f_-'(x,E;\hbar)$ are analytic in~$E$ in a disk
$|E|\les \hbar$. In particular, $\alpha(E;\hbar)$ is analytic around $E=0$ in the same neighborhood. Moreover, due to
$f_-(x,E;\hbar)=\overline{f_-(x,-E;\hbar)}$, one checks that $\Re f_-(x_0,E;\hbar)$ and  $\Im f_-(x_0,E;\hbar)$ are even
and odd in~$E$, respectively. Thus, it follows that $\Im\alpha(E;\hbar)$ is odd in~$E$, whereas $\Re\alpha(E;\hbar)$ is even. Moreover, for any $k,n\ge0$,
\begin{equation}
\label{eq:alphaderbds}
|\del_E^k \Im\alpha(E;\hbar)|\le C_{k,n} \hbar^n
\end{equation}
which follows from the fact that $S_-(E;\hbar)>0$ uniformly in~$-\eps<E<\eps$ as well as the differentiability properties of $S_-(E;\hbar)$ in~$E$, see~\cite{CST}. 
In view of these properties,
\begin{align}
e(E;x_{0},x_{0};\hbar)&=
\hbar\Im\Big[\frac{\alpha(E;\hbar)}{1+\tilde d_1\alpha(E;\hbar)-\alpha(E;\hbar)(\tilde d_1+\beta(E;\hbar))}\Big]\nn \\
&=  \hbar\Im\Big[\frac{\alpha(E;\hbar)}{1+ \tilde d_1\alpha(E;\hbar)}\Big]+ O_\R\big( e^{-2\hbar^{-1} S_+(E;\hbar)} \big) \label{eq:ImplusO}
\end{align}
Since $S_+(E;\hbar)\sim -\log E$ as $E\to0+$,
\[
O\big( e^{-2\hbar^{-1} S_+(E;\hbar)} \big) = O(E^{N}),\quad N=\hbar^{-1}
\]
Moreover, the imaginary part in~\eqref{eq:ImplusO}, i.e.,
\[
\eta(E;\hbar):=\Im\Big[\frac{\alpha(E;\hbar)}{1+\tilde d_1\alpha(E;\hbar)}\Big]
\]
 is an odd function in~$E$ (and analytic near $E=0$) and it satisfies the bounds
\[
|\del_E^k \eta(E;\hbar)|\le C_{k,n} \hbar^n
\]
cf.~\eqref{eq:alphaderbds}, and $\del_E^k \eta(0;\hbar)=0$ for even~$k$.
It is now easy to bound~\eqref{eq:semievol}: for any $n\ge0$, and all $t\ge0$, and any $0\le k\ll \hbar^{-1}$,
\begin{align*}
& \hbar^{-2} \Big|\int_0^\infty \sin\big(\hbar^{-1}tE\big) \Im \Big[
\frac{f_{+}(x_0,E;\hbar)f_{-}(x_0,E;\hbar)}{W(f_{+}(x_0,E;\hbar),f_{-}(x_0,E;\hbar))}
\Big] \chi_\eps(E)\,dE  \Big|\\
&\les  C_{k,n} \,\hbar^n\la t\ra^{-k}
\end{align*}
by integrating by parts (here $\chi_\eps$ is a smooth localizer to energies $E<\eps$). In other words, by taking $\hbar$ sufficiently
small one can achieve any rate of decay. Moreover, we note the important property that small energies do not present any kind of
obstruction to the problem  of summing in the angular momentum~$\ell$; in fact, the contributions of low lying energies to the decay
estimates decay rapidly in~$\ell$.

\subsection{The weighted $L^2$ bound on the spectral measure}
\label{subsec:eL2}

Here we generalize the analysis of Section~\ref{subsec:e00} to allow for general $x,x'$. More precisely,
we claim the following result which is a routine application of the basis representations which we have obtained above, cf.~
Section 8 in \cite{DSS}.

\begin{lemma}
 \label{lem:eL2}
 Let $0\le M\ll \hbar^{-1}$.
The spectral measure as defined in~\eqref{eq:esemiclass} satisfies the bounds,
\begin{equation}
 \label{eq:eL2}
\sup_{0<E<\eps} \| \la x\ra^{-k-\frac12-} \del_E^k e(E,x,x';\hbar)  \la x'\ra^{-k-\frac12-} \|_{L^2_{x,x'}} \le C_{k,n} \hbar^n
\end{equation}
for any $n\ge0$ and any $0\le k\le M$. Moreover, for any choice of $x,x'\in\R$, one has the property
\begin{equation}
 \label{eq:oddder}
\lim_{E\to0+} \del_E^{2j}e(E,x,x';\hbar)  =0
\end{equation}
for any $0\le j\le \frac{M}{2}$.
\end{lemma}

\subsection{The decay estimate for small energies}
\label{subsec:dispest_low}

It is now a simple matter to establish the desired decay estimate
for~\eqref{eq:semievol} for small energies.

\begin{prop}
\label{prop:displow}
Let $0\le M\ll\hbar^{-1}$ be given. Then for any $n\ge0$, $0\le k\le M$, and $t\ge0$
one has the bounds
\begin{align*}
 \Big\|\la x\ra^{-k-\frac12-}\int_0^\infty \sin\big(\hbar^{-1}tE\big) e(E,x,x';\hbar) \chi_\eps(E)\,dE \la x'\ra^{-k-\frac12-}
  \Big \|_{2\to2}
\le  C_{k,n} \,\hbar^n\la t\ra^{-k} \\
\Big \|\la x\ra^{-k-\frac12-}\int_0^\infty \cos\big(\hbar^{-1}tE\big) e(E,x,x';\hbar) \chi_\eps(E)\,EdE \la x'
\ra^{-k-\frac12-} \Big \|_{2\to2}
\le  C_{k,n} \,\hbar^n\la t\ra^{-k}
\end{align*}
uniformly in $0\le \hbar\le \hbar_0$.
\end{prop}
\begin{proof}
This follows directly from Lemma~\ref{lem:eL2} by repeated integrations by parts.
\end{proof}

We remark that these estimates immediately transfer to $L^1\to L^\infty$ bounds by means
of Bernstein's inequality. In fact, one can also establish the following result with the optimal $\la x\ra^{-k}$ weights by means of a more careful
treatment of the oscillatory integrals as in~\cite{DSS}.  

\begin{prop}
\label{prop:displow2}
Let $0\le M\ll\hbar^{-1}$ be given. Then for any $n\ge0$, $0\le k\le M$, and $t\ge0$
one has the bounds
\begin{align*}
 \Big\|\la x\ra^{-k}\int_0^\infty \sin\big(\hbar^{-1}tE\big) e(E,x,x';\hbar) \chi_\eps(E)\,dE \la x'\ra^{-k}
  \Big \|_{1\to\infty}
\le  C_{k,n} \,\hbar^n\la t\ra^{-k} \\
\Big \|\la x\ra^{-k}\int_0^\infty \cos\big(\hbar^{-1}tE\big) e(E,x,x';\hbar) \chi_\eps(E)\,EdE \la x'
\ra^{-k} \Big \|_{1\to\infty}
\le  C_{k,n} \,\hbar^n\la t\ra^{-k}
\end{align*}
uniformly in $0\le \hbar\le \hbar_0$.
\end{prop}

We call the reader's attention to the fact that these bounds {\em decay rapidly} with the angular momentum $\ell\simeq \hbar^{-1}$. In other
words, energies (in the original formulation of the Regge-Wheeler equation) of size $\ll \ell^{2}$ do {\em not present any obstruction} to the summation in~$\ell$. 
This is a reflection of the expectation that any such obstruction should result from the local behavior of the potential around the 
maximum due to complex resonances. In the related context of the
surfaces of revolutions, this corresponds to the principle that the growth of the constants $C(\ell)$ in the decay estimates of~\cite{SSS2} is determined by the local
geometry of the manifold rather than its asymptotic behavior at the ends. In particular, if the surface contains a large trapping set (such as an equatorial section of a sphere) then the constants grow exponentially in~$\ell$, 
rendering summation impossible. 

\section{Energies close to the top, Mourre and Sigal-Soffer estimates}
\label{eq:topscattering}

For  energies in the range $\eps<E<100$ we establish a
Mourre estimate which then allows us to invoke the semiclassical Sigal-Soffer type decay bounds of Section~\ref{sec:HSS}. Thus,
let $p:= -i\hbar \del_x$, $H:=p^2+V$ as above, and
$A:=\frac12(px+xp)$. Note that the Mourre estimate is shown here to hold in a
neighborhood of a {\em trapping energy} (namely, $E=1$). For notational convenience, we 
shift the location of the maximum to $\xphot=0$ in this section.

\begin{lemma}
\label{lem:Mourre}
 For $\eps>0$ and $\hbar$ small, there exists a fixed constant $c_0>0$ so that
\begin{equation}\label{eq:Mourre}
 \chi_I(H) \frac{i}{\hbar} [H,A]\chi_I(H) \ge c_0 \hbar \chi_I(H)
\end{equation}
where $\chi_I$ is the indicator of $I:=[\eps/2,100]$.
\end{lemma}
\begin{proof}
We split $I=I_0\cup I_1$ where $I_0:=[\eps/2,1-\eps/2]$ and $I_1:=[1-2\eps,100]$.
We start with the latter, and write $I$ instead of $I_1$ for simplicity.
 First,
\begin{equation}\label{eq:HAbrack}
  \frac{i}{\hbar} [H,A] = 2p^2 - xV'(x;\hbar)\ge p^2 - xV'
\end{equation}
Hence, with $g_I$ being a smooth cutoff function adapted to $I$,
\begin{align*}
&g_I(H) \frac{i}{\hbar} [H,A]g_I(H) \\&\ge g_I(H) (p^2-xV')g_I(H)\\
&\ge g_I(H) ((p^2-xV')F^2 +  F^2 (p^2-xV') + (p^2-xV') \bar F^2 +\bar F^2 (p^2-xV')) g_I(H)\\
&\ge g_I(H) \big[ 2F(p^2-xV')F + [F,[F,p^2]] + \bar F^2 H \tilde g_I(H) + H \tilde g_I(H) \bar F^2  \\
&\quad
  + 2\bar F^2 (-xV' - V)\big] g_I(H)
\end{align*}
Here $1=F+\bar F$ is a smooth partition of unity with $F(x)=1$ on $[-x_1,x_1]$ where $x_1>0$ will be a large number depending only on~$V$.
Moreover, $\tilde g_I$ is another function adapted to $I$ with $\tilde g_I g_I=g_I$.
By \eqref{eq:nondegmax} $F(-xV')F\ge cx^2 F^2$ for some $c>0$ depending on $x_1$ and the Heisenberg uncertainty principle implies that
\[
 F(p^2-xV')F\ge c\, F(p^2+x^{2})F\ge 2c_0\hbar F^2
\]
The uncertainty principle here is being used in the form
\[
\| p \psi\|_{2}^{2}+\|x\psi\|_{2}^{2}\ge 2\|p\psi\|_{2}\|x\psi\|_{2}\ge {\hbar}\|\psi \|_{2}^{2}
\]
which immediately follows from the fact that $[p,x]=-i\hbar$, see for example~\cite{GS}. 
Furthermore,
\[
 [F,[F,p^2]]=-2\hbar^2 (F')^2
\]
and
\begin{align*}
 &g_I (\bar F^2 H\tilde g_I + H\tilde g_{I}\bar F^2) g_{I} \\&= g_I (\bar F^2 (H-1)\tilde g_I +
 (H-1)\tilde g_I \bar F^2 + (\bar F^2 \tilde g_I + \tilde g_I \bar F^2))g_I \\
&= g_I (2\bar F (H-1)\tilde g_I \bar F + 2\bar F\tilde g_I \bar F + [\bar F,[\bar F,(H-1)\tilde g_I]]
 + [\bar F,[\bar F, \tilde g_I ]] )g_I \\
& \ge g_I (2(1-\eps) \bar F^2 - O(\hbar^2))g_I \ge g_I \bar F^2 g_I  - C\hbar^2\, g_I  F^2 g_I 
\end{align*}
where we used that $\|[\bar F, [\bar F,\tilde g]]\|\les \hbar^2$, see
Lemma~\ref{lem:comm} below. Finally, from the shape of our potential
$V(x;\hbar)$ one verifies easily that $x_1$ can be chosen such that
$-xV'-V\ge0$ for all $|x|\ge x_1$ whence $\bar F^2(-xV'-V)\ge0$. In
view of the preceding,
\[
 g_I(H) \frac{i}{\hbar} [H,A]g_I(H)  \ge c_0 \hbar g_I^2(H)
\]
as desired.
Finally, on the interval $I_0$  one can use \eqref{eq:HAbrack} directly since one has a classical nontrapping condition
on energies in that range. This then gives the desired Mourre estimate in that range of energies, see Theorem~1 in~\cite{Graf}.
\end{proof}

The following commutator bound was used in the previous proof.

\begin{lemma}
 \label{lem:comm} Let $F$ and $g$ be smooth and compactly supported. Then \[ \|[F(x), [F(x), g(H)]]\|\le C \hbar^2 \] where $C=C(F,g)$.
\end{lemma}
\begin{proof} For simplicity we show that $\|[F(x), g(H)]\|\le C \hbar$, the double commutator being an obvious variation thereof. 
 By the commutator expansion formula~\eqref{eq:comm_exp} one has
\[
 \| [F(x), g(H)] \| \les C(F) \| [x,g(H)] \|
\]
Now $g(H) = \tilde g(\tilde H)$ where $\tilde H := H (H+1)^{-1}$ and $\tilde g$ is again smooth and compactly supported. Hence one
can expand with the bounded $\tilde H$ to conclude that
\[
 \| [F(x), g(H)] \| \les C(F,g) \| [x,\tilde H] \| = C(F,g) \| (H+1)^{-1} [x,H] (H+1)^{-1} \| \les C(F,g)\hbar
\]
Here we used that $[x,H]= -2i \hbar p$ and $\|p(H+1)^{-1}\|\les \|p(1+p^2)^{-1}\|\les  1$.
\end{proof}

Since we are dealing with wave rather than the Schr\"odinger equation, we need to derive a Mourre estimate for $\sqrt{H}$ rather than~$H$.
However, this is an easy consequence of the Kato square root formula. 

\begin{cor}\label{cor:Mourre}
 For $\eps>0$ and $\hbar$ small, there exists a fixed constant $\tilde c_0>0$ so that
\begin{equation}\nn 
 \chi_I(H) \frac{i}{\hbar} [\sqrt{H},A]\chi_I(H) \ge \tilde c_0 \hbar \chi_I(H)
\end{equation}
where $\chi_I$ is the indicator of $I:=[\eps/2,100]$.
\end{cor}
\begin{proof}   One uses that 
\[
H^{-\frac12}\chi_I(H) = \frac{1}{\pi}\int_0^\infty (H+\lambda)^{-1}\lambda^{-\frac12}\, d\lambda\; \chi_I(H)
\]
whence by Lemma~\ref{lem:Mourre}, 
\begin{align*}
&\chi_I(H) \frac{i}{\hbar} [\sqrt{H},A]\chi_I(H) \\
&= \chi_I(H)\sqrt{H}  \frac{i}{\hbar} [A, H^{-\frac12}] \sqrt{H} \chi_I(H) \\
&= \frac{1}{\pi} \chi_I(H)\sqrt{H}  \int_0^\infty     \frac{i}{\hbar}[A, (H+\lambda)^{-1} ]  \lambda^{-\frac12}\, d\lambda\; \sqrt{H}\chi_I(H)\\
&= \frac{1}{\pi}\sqrt{H}  \int_0^\infty    (H+\lambda)^{-1}  \chi_I(H) \frac{i}{\hbar} [ H, A ]   \chi_I(H) (H+\lambda)^{-1}\lambda^{-\frac12}\, d\lambda\; \sqrt{H}\\
&\ge \frac{1}{\pi}\sqrt{H}  \int_0^\infty    (H+\lambda)^{-1} c_0\,\hbar \chi_I(H) (H+\lambda)^{-1}\lambda^{-\frac12}\, d\lambda\; \sqrt{H}\\
&\ge \tilde c_0 \hbar \chi_I(H)
\end{align*}
and we are done. 
\end{proof}

In order to apply the time-decay result from Section~\ref{sec:HSS}, we need to verify the basic commutator assumption~\eqref{eq:comm_ass}.
For the definition of $\ad_{A}^{k}$ we refer the reader to that section. 

\begin{lemma}
For any smooth function $g$ on the line with support in $(0,\infty)$ one has 
\[
 \|\ad_A^k(g(\sqrt{H}))\|\le C(k,g) \hbar^k
\]
for all $k\ge1$. 
\end{lemma}
\begin{proof}
For the purposes of this proof, we call any smooth function $g$ on the line with support in $(0,\infty)$ {\em admissible}. 
First, there exists another admissible function $\tilde g$ with $g(\sqrt{H})=\tilde g(H)$. Second, with $\tilde H=H(H+1)^{-1}$ 
for any admissible $g$ there exists $\tilde g$ admissible such that 
 $g(H)=\tilde g(\tilde H)$. So it suffices to consider $\ad_{A}^{k}(g(\tilde H))$ with admissible~$g$.  

As a preliminary calculation, note that 
\begin{align*}
 i[H,A] &= \hbar(2H-(2V+xV')) =: \hbar (2H+V_{1}) \\
  i[\tilde H,A] &= (H+1)^{-1} i[H,A] (H+1)^{-1} = \hbar (H+1)^{-1} (2H+V_{1}) (H+1)^{-1}
\end{align*}
whence $\| [\tilde H,A]\| \le C\, \hbar$. At the next level, 
\begin{align*}
i[i[H,A]A] &= \hbar (2 i [H,A] + i [V_{1},A]) = \hbar^{2} ( 4H + 2V_{1} - xV_{1}')   
\end{align*}
For $\tilde H$ we use the general identity 
\[
[SBS, A]= S B[S,A] + S[B,A]S + [S,A]BS
\]
to conclude that 
\begin{multline}\label{eq:comm 2}
i[i[\tilde H,A]A] =  i \hbar \big\{ (H+1)^{-1} (2H+V_{1}) [(H+1)^{-1},A] +  (H+1)^{-1} [2H+V_{1},A] (H+1)^{-1} \\
 + [(H+1)^{-1}  ,A]  (2H+V_{1})  (H+1)^{-1} \big\}
\end{multline} 
Inserting
\[
i[(H+1)^{-1},A] = -(H+1)^{-1} i[H,A](H+1)^{-1}
\]
into \eqref{eq:comm 2} implies that  $\| i[i[\tilde H,A]A]\| \le C\hbar^{2}$.  Continuing in this fashion implies 
\[
\| \ad^{k}_{A}(\tilde H)\|\le C(k)\, \hbar^{k}\qquad\forall\; k\ge1
\]
Next, we transfer this estimate to $\ad_{A}^{k}(g(\tilde H))$ via an {\em almost analytic extension} of 
an admissible function~$g$. This refers to a smooth function $G_{N}(z)$ in the complex plane of compact support
such that $g=G_{N}$ on the real axis and  with 
\begin{equation}
\label{eq:almost analytic}
|(\partial_{\bar z}G_{N})(z)|\le C_{N}\, |\Im z|^{N}
\end{equation}
for an arbitrary but fixed positive
integer~$N$. One then has  the {\em Helffer-Sj\"ostrand formula}
\begin{equation}\label{eq:Helffer}
g(\tilde H) = \frac{1}{\pi} \int_{\C} (\partial_{\bar z}G_{N})(z) (\tilde H-z)^{-1}\, m(dz)
\end{equation}
 where $m$ is the Lebesgue measure on~$\C$, see~\cite[Chapter 2]{Davies}.   The desired estimate now follows from 
 \[
 \ad_{A}^{k}(g(\tilde H)) = \frac{1}{\pi} \int_{\C} (\partial_{\bar z}G_{N})(z) \ad_{A}^{k}((\tilde H-z)^{-1})\, m(dz)
 \]
For example, for $k=1$ 
\[
\ad_{A}^{1}((\tilde H-z)^{-1}) = -(\tilde H-z)^{-1}[\tilde H,A](\tilde H-z)^{-1}
\]
and therefore 
\[
\big\| \ad_{A}^{1}((\tilde H-z)^{-1})\big\| \le C|\Im z|^{2} \hbar 
\]
Inserting this into~\eqref{eq:Helffer} and using~\eqref{eq:almost analytic} yields 
\[
\| [g(\tilde H), A]\| \le C\hbar
\]
The cases of higher $k$ are analogous. The larger $k$ is, the larger $N$ needs to be. 
\end{proof}

We are now ready to state the main decay estimate for intermediate energies. 

\begin{cor}
 \label{cor:decay_top} One has for small $\hbar$ and all $t\ge0$, as well as any $\alpha\ge0$,
\begin{equation}
 \label{eq:decaytop22} \| \la x\ra^{-\alpha} e^{i\frac{t\sqrt{H}}{\hbar}} \chi_I(H) \la x\ra^{-\alpha}\|_{2\to2} \le C(\alpha)   \la \hbar t\ra^{-\alpha}
\end{equation}
Furthermore,
\begin{equation}
 \label{eq:decaytop1infty} \| \la x\ra^{-\alpha} e^{i\frac{t\sqrt{H}}{\hbar}} \chi_I(H) f\|_\infty \le C(\alpha) \hbar^{-1} \la \hbar t\ra^{-\alpha}
\| \la x\ra^{\alpha}f\|_1
\end{equation}
\end{cor}
\begin{proof} 
By Corollary~\ref{cor:Mourre} and the previous lemma, we conclude from Proposition~\ref{prop:HSS}
that for any admissible function $g$ (as defined in the previous proof) and any $\alpha\ge0$
\begin{equation}
\label{eq:deca 1}
\| \la   A\ra^{-\alpha} e^{-i\frac{t\sqrt{H}}{\hbar}}
  g(H) \la   A\ra^{-\alpha} f\|_{2} \le C\, \la \hbar t\ra^{-\alpha} \| f\|_{2}
\end{equation}
To derive \eqref{eq:decaytop22} from this estimate, we pick another admissible $\tilde g$ so that $\tilde g(\tilde H)g(H)=g(H)$ where $\tilde H=H(H+1)^{-1}$
as before.  Moreover, the support of $\tilde g$ is taken to lie strictly within~$(0,1)$. Then 
\begin{align*}
 \la x\ra^{-\alpha} e^{i\frac{t\sqrt{H}}{\hbar}} \chi_I(H) \la x\ra^{-\alpha} &=    \la x\ra^{-\alpha} \tilde g(\tilde H) \la   A\ra^{\alpha} \,   \la   A\ra^{-\alpha}  e^{i\frac{t\sqrt{H}}{\hbar}} \chi_I(H) \la   A\ra^{-\alpha}  \, \la   A\ra^{\alpha} \tilde g(\tilde H) \la x\ra^{-\alpha}
\end{align*}
It therefore suffices to prove that   
\begin{equation}\label{eq:Cal g}
\big\| \la x\ra^{-\alpha} \tilde g(\tilde H) \la   A\ra^{\alpha} \big\| \le C(\alpha)
\end{equation}
The logic here is that  the cutoff $\tilde g(\tilde H)$ guarantees that $H=p^{2}+V$ is bounded, whence also $p^{2}$ is bounded. But then~$p$ is bounded, so $A$ should be at
most as large as~$x$ which justifies~\eqref{eq:Cal g}. By complex interpolation, it suffices to prove that~\eqref{eq:Cal g} holds for positive integers~$\alpha$. Moreover, composing with the
 adjoints shows that this is the same as 
 \[
 \big\| \la x\ra^{-\alpha} \tilde g(\tilde H) \la   A\ra^{2\alpha}   \tilde g(\tilde H)  \la x\ra^{-\alpha}   \big\| \le C(\alpha)^{2}
 \]
For example, set $\alpha=1$. Then one checks that 
\begin{align*}
\la   A\ra^{2} &=1 + \frac14(px+xp)^{2} = 1-\frac{\hbar^{2}}{4} + xp^{2}x 
=  1-\frac{\hbar^{2}}{4} -x V x + xHx
\end{align*}
Since $V=O(\la x\ra^{-2})$, it suffices to bound $xHx$.    Let $G$ denote the almost analytic extension of~$\tilde g$ as in the proof of the previous lemma. 
Then
\begin{align}
& \la x\ra^{-1}\tilde g(\tilde H) xHx \tilde g(\tilde H) \la x\ra^{-1}  = \nn \\
& = \la x\ra^{-1} ( x\tilde g(\tilde H)  + [\tilde g(\tilde H), x]) H ( \tilde g(\tilde H) x  -  [\tilde g(\tilde H), x]) \la x\ra^{-1}  \label{eq:G H x}
\end{align}
It is clear that the terms involving no commutators are bounded. 
For the commutators in the second line we use the Helffer-Sj\"ostrand formula as before, viz. 
\begin{align}
[\tilde g(\tilde H), x] &= \frac{1}{\pi}\int_{\C} \partial_{\bar z}G(z) [(\tilde H-z)^{-1}, x] \, m(dz) \nn  \\
&= \frac{1}{\pi}\int_{\C} \partial_{\bar z}G(z)(\tilde H-z)^{-1}  [x,\tilde H] (\tilde H-z)^{-1} \, m(dz) \nn   \\
&= \frac{1}{\pi}\int_{\C} \partial_{\bar z}G(z)(\tilde H-z)^{-1}  (H+1)^{-1} (-2i\hbar p) (H+1)^{-1}(\tilde H-z)^{-1} \, m(dz)   \label{eq:HS xH}
\end{align}
In particular, $[\tilde g(\tilde H), x]$ is a bounded operator. 
Inserting this into~\eqref{eq:G H x} concludes the argument for~$\alpha=1$. For $\alpha>1$ the argument is similar. We begin by expanding
for $\ell\ge1$ an integer 
\begin{align}
\la A\ra^{2\ell} &= (1+ (xp+px)^2/4)^{\ell} \nn \\
&=  \sum \mathrm{const}\cdot x^{m_{1}} p^{n_{1}}x^{m_{2}} p^{n_{2}}\cdots  x^{m_{s}} p^{n_{s}}   \label{eq:Aell}
\end{align}
where the sum extends over integer $m_{i}, n_{i}$ with 
\[
\sum_{i} n_{i} \le 2\ell,\qquad \sum_{i} m_{i} \le 2\ell
\]
Moreover,  using the commutator $[p,x]=-i\hbar$ to move powers of~$p$ through powers of~$x$, the general term in~\eqref{eq:Aell}
may be written as $x^{k} p^{2k} x^{k}$ where $k\le \ell$. Hence, we need to show that 
\begin{equation}\label{eq:xjp}
\la x\ra^{-\ell}  \tilde g(\tilde H) x^{k} p^{2k} x^{k} \tilde g(\tilde H) \la x\ra^{-\ell}
\end{equation}
with $0\le k\le \ell$ is a bounded operator.   First, the operator in~\eqref{eq:xjp} is nonegative, and moreover bounded above by 
\begin{equation}\label{eq:xjp 2}
\la x\ra^{-\ell}  \tilde g(\tilde H) x^{k} H^{k} x^{k} \tilde g(\tilde H) \la x\ra^{-\ell}
\end{equation}
since $p^{2}\le p^{2}+V=H$. Note that if we can move $x^{k}$ across the spectral cut-offs, then we are done since $0\le k\le\ell$. To accomplish this, 
we start from 
the following identity, which is proved by induction: for every $k\ge2$ 
\begin{equation}\label{eq:xk H commute}
[x^{k},H] = - 2i\hbar \sum_{j=1}^{k-1} x^{k-j-1} p x^{j}
\end{equation}
and $[x,H]=-2i\hbar p$. Several comments are in order: first, 
domain considerations are irrelevant due to the cutoff $\tilde g(\tilde H)$ which is always applied. In fact, we may use this formally and in the
end justify the procedure a posteriori by obtaining a bound on the $L^{2}$-operator norm. Second, the total weight in~$x$ on the right-hand side of~\eqref{eq:xk H commute} is $k-1$. 
And third, in any given term $ x^{k-j-1} p x^{j}$ we can shift the position of~$p$ arbitrarily using the commutator~$[p,x]=-2i\hbar$. 
To proceed, one has 
\[
[ x^{k}, \tilde H] = (H+1)^{-1} [x^{k}, H] (H+1)^{-1} 
\]
so that 
\begin{equation}
\label{eq:HS xk H}
[\tilde g(\tilde H), x^{k}]  = \frac{1}{\pi}\int_{\C} \partial_{\bar z}G(z)(\tilde H-z)^{-1} (H+1)^{-1}  [x^{k}, H] (H+1)^{-1} (\tilde H-z)^{-1} \, m(dz) 
\end{equation}
Inserting~\eqref{eq:xk H commute} into the right-hand side of~\eqref{eq:HS xk H} and in view of the preceding comments we arrive at an expression of the form 
\[
\frac{1}{\pi}\int_{\C} \partial_{\bar z}G(z)(\tilde H-z)^{-1} (H+1)^{-1}  x^{k-1}p (H+1)^{-1} (\tilde H-z)^{-1} \, m(dz) 
\]
If $k-1=0$ we are done since $p(H+1)^{-1}$ is bounded. Otherwise, commuting $x^{k-1}$ through $(H+1)^{-1} $ to the left reduces the weight by another power.
In other words, one obtains $x^{k-2}$. Because of this reduction of the degree, the process must terminate after at most~$k$ commutations, and we are done
with the proof of the first estimate~\eqref{eq:decaytop22}.

Heuristically speaking, 
 the second bound~\eqref{eq:decaytop1infty} is derived from the first by means of the following principle, known as 
  Bernstein's inequality:  if $\vphi\in L^{2}(\R)$ satisfies 
 $\supp(\hat{\vphi})\subset [-R,R]$ (with $\hat{\vphi}$ being the Fourier transform), then $\vphi\in L^{\infty}(\R)$ with the bound 
 \[
 \|\vphi\|_{\infty} \le \| \hat{\vphi}\|_{1} \le (2R)^{\frac12}\|\vphi\|_{2}
 \]
 where the second inequality is obtained by Cauchy-Schwartz followed by Plancherel's theorem. 
 
 To see the relevance of this, 
 let $g_{I}(H)$ with $g_{I}$ smooth be as above. 
 Since $p^2+V\le 100$
 on the support of $g_I(H)$, one sees -- again at least heuristically --  that also 
 $\del_x^2\le 100\hbar^{-2}$ which restricts the Fourier support to size~$\le C\hbar^{-1}$. These operator inequalities
 can be interpreted in the sense of positive operators, or via quadratic forms, say.   Ignoring the distinction between $H$ and 
  the ``free ''case in which $H=H_{0}:=p^{2}$, we 
  obtain    via Bernstein that 
 \begin{equation}\label{eq:hbar -1/2 loss}
\| \la x\ra^{-\alpha} g_I(H) f\|_\infty \les \hbar^{-\frac12} \|
\la x\ra^{-\alpha}f\|_2
 \end{equation}
  Replacing $L^2$ on the right-hand side costs another $\hbar^{-\frac12}$ by duality, so that one loses $\hbar^{-1}$ in total over the $L^{2}$-bound, 
  which is what~\eqref{eq:decaytop1infty} claims. 
Note that we passed the weight in~$x$ through $g_{I}(H)$ onto $f$ which is another technical issue, next to the distinction between $H$ and~$H_{0}$. 

In oder to rigorously implement these ideas it is advantageous to work with resolvents rather than the (distorted) Fourier transform.  To be specific, we write
\begin{multline}\label{eq:res H}
\la x\ra^{-\alpha} e^{i\frac{t\sqrt{H}}{\hbar}} \chi_I(H) \la x\ra^{-\alpha} \\
= \la x\ra^{-\alpha} (1+H)^{-1} \la x\ra^{\alpha}   \la x\ra^{-\alpha}  e^{i\frac{t\sqrt{H}}{\hbar}}(1+H)^{2} \chi_I(H) \la x\ra^{-\alpha}  \la x\ra^{\alpha} (1+H)^{-1} \la x\ra^{-\alpha}  
\end{multline}
Note that $(1+H)^{2} \chi_I(H)$ is just another cut-off. Therefore,  the $L^{2}$-decay bound applies to 
\[
\la x\ra^{-\alpha}  e^{i\frac{t\sqrt{H}}{\hbar}}(1+H)^{2} \chi_I(H) \la x\ra^{-\alpha} 
\]
and it suffices to prove that 
\begin{equation}\label{eq:2 infty}
\| \la x\ra^{-\alpha} (1+H)^{-1} \la x\ra^{\alpha} f \|_{\infty} \le C(\alpha) \hbar^{-\frac12}\|f\|_{2}
\end{equation}
which by duality then implies the corresponding $L^{1}\to L^{2}$ estimate and thus implies~\eqref{eq:decaytop1infty}.
To prove~\eqref{eq:2 infty} we represent the Green function, i.e., the kernel of $(1+H)^{-1}$,  in the form
\begin{equation}\label{eq:Green}
(1+H)^{-1}(x,x')= \hbar^{-2} \frac{\psi_{+}(x)\psi_{-}(x')}{W(\psi_{+},\psi_{-})} \chi_{[x>x']}+\hbar^{-2} \frac{\psi_{-}(x)\psi_{+}(x')}{W(\psi_{+},\psi_{-})} \chi_{[x<x']}
\end{equation}
with $W$ denoting the Wronskian, and 
where $\psi_{\pm}$ are the Jost solutions to $1+H$ which are defined uniquely by 
\begin{align*}
-\hbar^{2} \psi_{\pm}'' + V\psi &= \psi_{\pm} \\
\psi_{\pm}(x) &\sim e^{\mp \frac{x}{\hbar}} \text{\ \ as\ \ }x\to\pm \infty
\end{align*}
These solutions are given in terms of Volterra integral equations in the form
\begin{equation}\label{eq:psi+}
\psi_{+}(x)= \psi_{+,0}(x) - \hbar^{-1}\int_{x}^{\infty} e^{\frac{x-y}{\hbar}} \, V(y)\psi(y)\, dy
\end{equation}
where $\psi_{\pm,0}:=e^{\mp\frac{x}{\hbar}}$ 
and symmetrically for $\psi_{-}$. 
By the maximum principle (or elementary convexity arguments - recall that $V>0$) one sees that $\psi_{\pm}>0$
on the line.  In view of~\eqref{eq:psi+} therefore $0<\psi_{\pm }<\psi_{\pm,0}$ and the Green function in~\eqref{eq:Green}
satisfies
\[
0< (1+H)^{-1}(x,x') \le C\hbar^{-1} e^{-\frac{|x-x'|}{\hbar}} 
\]
We used here that $W(\psi_{+},\psi_{-})\ge c\hbar^{-1}$ which follows by differentiating and/or evaluating~\eqref{eq:psi+} at $x=0$. 
In conclusion, in order to prove~\eqref{eq:2 infty} we need to show that the kernel 
\[
  \hbar^{-1}  \la x\ra^{-\alpha}  e^{-\frac{|x-x'|}{\hbar}}      \la x'\ra^{\alpha} 
\]
is bounded as an operator from $L^{2}\to L^{\infty}$ with norm $\le C\hbar^{-\frac12}$. But this follows from Cauchy-Schwarz and we are done. 
 \end{proof}

\section{Large energies}

This is comparatively easier than the other two regimes of energies. Indeed,
the energy $E$ is so much larger than the potential that the free case becomes dominant.
Technically speaking, we use the classical WKB ansatz without turning points.

\subsection{The WKB ansatz for large energies}
\label{subsec:WKBlargeE}

We shall use the outgoing Jost solutions $f_+(x,E;\hbar)$ which are defined uniquely as solutions to the equations
\begin{align*}
-\hbar^{2} f_+''(x,E;\hbar) + Vf_+(x,E;\hbar) & = E^{2} f_+(x,E;\hbar) \\
f_+(x,E;\hbar) &\sim e^{\pm i\frac{E}{\hbar}x } \text{\ \ as\ \ }x\to\pm\infty
\end{align*}
A global (at least on $x\ge0$) representation of $f_+(x,E;\hbar)$ is given by the WKB  ansatz
\begin{equation}\label{eq:WKB E large}
f_+(x,E;\hbar) = E^{\frac12} e^{\frac{i}{\hbar}T_+(E;\hbar)} Q^{-\frac14}(x,E;\hbar) e^{\frac{i}{\hbar}\int_0^x \sqrt{Q(y,E;\hbar)}\, dy} (1+\hbar a_+(x,E;\hbar))
\end{equation}
where $Q(x,E;\hbar):= E^2-V(x;\hbar)$ and
\[
T_+(E;\hbar):= \int_0^\infty \big(E-\sqrt{Q(y,E;\hbar})\big)\, dy
\]
The prefactor $E^{\frac12}$ is a convenient normalization, and $T_{+}$ guarantees the correct asymptotics at $x=+\infty$. 
This representation is valid for $x\ge0$, which is justified by the bounds
\begin{equation}
\label{eq:adecay}
|a_+(x,E;\hbar)| \les \la x\ra^{-3} E^{-2} \quad\forall E\ge 100,\;x\ge0
\end{equation}
To obtain these estimates we start from the following equation for~$a(x)$, which is obtained by inserting the ansatz~\eqref{eq:WKB E large} into the defining equation
for $f_{+}$:
\begin{equation}\label{eq:a eq}
\hbar (\psi^2 \dot a)^{\dot{\ }} = -\psi^2 V_2 (1+\hbar a),\quad a(\infty,E;\hbar)=\dot a(\infty,E;\hbar)=0
\end{equation}
where
\[
\psi(x):= Q^{-\frac14}(x,E;\hbar) e^{\frac{i}{\hbar}\int_0^x \sqrt{Q(y,E;\hbar)}\, dy}
\]
and
\begin{align*}
V_2(x) &= \frac{5}{16} \Big(\frac{\dot Q(x)}{Q(x)}\Big)^2 - \frac14 \frac{\ddot Q(x)}{Q(x)}\\
&=  \frac{5}{16} \Big(\frac{\dot V(x)}{E^2-V(x)}\Big)^2 + \frac14 \frac{\ddot V(x)}{E^2-V(x)} = O(E^{-2}\la x\ra^{-4})
\end{align*}
using that $V$ decays at least as fast as an inverse square. 
The solution of~\eqref{eq:a eq} is uniquely given in terms of the 
Volterra integral equation
\[
a(x,E;\hbar) = \int_x^\infty \big(1-e^{\frac{i}{\hbar}\int_x^y \sqrt{Q(u,E;\hbar)}\,du}\big)V_2(y,E;\hbar)(1+\hbar a(y,E;\hbar))\,dy
\]
In addition to~\eqref{eq:adecay}, this integral equation implies the derivative bounds
\begin{equation}\label{eq:a der jk} 
|\del_E^k \del_x^j a_+(x,E;\hbar)| \les \la x\ra^{-3-j} E^{-2-k} \quad\forall E\ge 100,\;x\ge0
\end{equation}
and all $k\ge0, j\ge0$.  While these statements are routine, we now give some indication on how they are obtained. Write
\begin{align*}
a(x,E;\hbar) &= a_{0}(x,E;\hbar) + \hbar \int_x^\infty  k(x,y;\hbar,E)  V_2(y,E;\hbar)  a(y,E;\hbar)\,dy  \\
a_{0}(x,E;\hbar) &:=  \int_x^\infty  k(x,y;\hbar,E)  V_2(y,E;\hbar)  \,dy \\
k(x,y;\hbar,E) &:= 1-e^{\frac{i}{\hbar}\int_x^y \sqrt{Q(u,E;\hbar)}\,du} 
\end{align*}
To see that \eqref{eq:a der jk}  holds for $a_{0}$ we expand the defining integral of~$a_{0}$ as follows:
\begin{multline*}
a_{0}(x,E;\hbar) = \int_{x}^{\infty} V_{2}(y,E;\hbar)  \,dy
 +i\hbar \int_{x}^{\infty}  \partial_{y} \Big[e^{\frac{i}{\hbar}\int_x^y \sqrt{Q(u,E;\hbar)}\,du}\Big] \frac{ V_2(y,E;\hbar)}{\sqrt{Q(y,E;\hbar)}}  \,dy \\
 = \int_{x}^{\infty} V_{2}(y,E;\hbar)  \,dy -i\hbar\frac{ V_{2}(x,E;\hbar)}{\sqrt{ Q(x,E;\hbar)}  } - 
 i\hbar \int_{x}^{\infty} \!\!\! e^{\frac{i}{\hbar}\int_x^y \sqrt{Q(u,E;\hbar)}\,du}\:  \partial_{y} \Big[ \frac{ V_2(y,E;\hbar)  }{\sqrt{Q(y,E;\hbar)} }\Big] \,dy 
\end{multline*}
The first two terms here satisfy the bounds~\eqref{eq:a der jk} by inspection, whereas the integral involving the oscillatory kernel needs to be expanded
further depending on the number of  derivatives, i.e., the size of $j+k$. Note that each further expansion improves the decay of the integrand by one power of~$E$
and~$y$, respectively. 

\subsection{Decay estimates in the regime of large energies}

The WKB considerations of Section~\ref{subsec:WKBlargeE}  imply the following decay estimate. For the definition of the spectral measure $e(E,x,x';\hbar)$
see the low energies regime. 

\begin{lemma}
\label{lem:largeEdecay}
Let $\chi_{>100}(E)$ be a smooth cutoff function supported in $(100,\infty)$ and equal to $1$ on $(200,\infty)$. Then for all $t>0$,
\begin{align*}
& \sup_{x\in\R}\Big|\la x\ra^{-k} \hbar^{-2}\int_{\R}\int_{0}^\infty \cos(\hbar^{-1} tE) 
e(E,x,x';\hbar)\,E \chi_{>100}(E) \, dE\; \la x'\ra^{-k} f(x')\,dx'\Big|\\ &\le C \hbar^{{-2}}\la t\ra^{-k} \int (|f'(y)|+|f(y)|)\, dy \\
&\sup_{x\in\R}\Big|\la x\ra^{-k} \hbar^{-1}\int_{\R}\int_{0}^\infty \sin(\hbar^{-1} tE) e(E,x,x';\hbar) \chi_{>100}(E) \, dE\; \la x'\ra^{-k} f(x')\,dx'\Big| \\&\le C \hbar^{-1}\la t\ra^{-k} \int |f(y)|\, dy
\end{align*}
Moreover, the same bounds hold as weighted $L^2\to L^2$ estimates, but with $\la \cdot\ra^{-k-\frac12-}$ instead of $\la\cdot\ra^{-k}$. 
\end{lemma}
\begin{proof}
This is essentially the same as in Section~9 of~\cite{DSS}. The only difference being the factor~$\hbar$.
However, we leave it to the reader to check that the proofs in~\cite{DSS} easily carry over to this case
as well. 
As for the $L^2\to L^2$ bounds, they follow from the $L^1\to L^\infty$ ones  by means of H\"older's inequality.
\end{proof}

\section{The proof of Theorems~\ref{thm:main} and~\ref{thm:main2}}\label{sec:proof}

We begin by reducing general data to those of fixed angular momentum. Thus
\begin{equation}
 \psi_0(x,\omega) = \sum_{\ell=0}^\infty \sum_{-\ell\le j\le \ell}
  \la \psi_0(x,\cdot), Y_{\ell,j}(\cdot)\ra_{S^2} Y_{\ell,j}(\omega)
\label{eq:Yelldecomp}
\end{equation} 
where $\{Y_{\ell,j}\}_{j=-\ell}^\ell$ is the usual orthonormal basis  of spherical harmonics in the space of $Y\in C^\infty(S^2)$ with $-\Delta_{S^2} Y= \ell(\ell+1)Y$. 
One has $\|Y_{\ell,j}\|_\infty \le C\la \ell\ra^{\frac12}$  where $C$ is an absolute constant. 
 Now let $Y$ be a normalized spherical harmonic with $-\Delta_{S^2}Y=\ell(\ell+1)Y$, and set $\hbar=\ell^{-1}$.  Consider 
 data $\psi[0]=(f,g)Y=(\psi_0,\psi_1)$. 
Let $\psi(t)$ denote the evolution of~$\psi[0]$ under the wave equation~\eqref{eq:wave}, as given by~\eqref{eq:semievol} and~\eqref{eq:sineevol}. Then by 
Lemma~\ref{lem:largeEdecay}, Corollary~\ref{cor:decay_top}, and Proposition~\ref{prop:displow}, provided $\ell$ is large, one obtains 
\begin{align}
& \| \la x\ra^{-k-\frac12-}  \psi(t) \|_{L^2(\R;L^2(S^2))} \nn \\&\les \la t\ra^{-k} \hbar^{-k-1}  \| \la x\ra^{k+\frac12+} (\hbar^{-1} \del_x\psi_0, \hbar^{-1} \psi_0,  \psi_1) \|_{L^2(\R;L^2(S^2))} \nn 
\end{align}
for any $0\le k\ll \ell$ and $t\ge0$.  Starting from general data $\psi[0]=(\psi_0,\psi_1)$, performing a decomposition as in~\eqref{eq:Yelldecomp}
we may sum up the $L^2$-bound over $\ell\gg k$, whereas for the finitely many remaining~$\ell$ we invoke the decay estimates from~\cite{DS} (for $\ell=0$)
and~\cite{DSS} (for $\ell>0$). In this way one obtains~\eqref{eq:decaywaveL2}. The reason why $\la x\ra^{-\frac92-}$ weights are required stems from the fact the
corresponding $L^1\to L^\infty$ bounds in~\cite{DS} and~\cite{DSS} need $\la x\ra^{-4}$ for $t^{-3}$ decay, and then we lose another $\la x\ra^{-\frac12-}$ due to H\"older's
inequality. On another technical note, the weights $\la x\ra^{-k-\frac12-}$ for $k=3$ (as in our case) essentially retain the orthogonality properties of the spherical harmonics which allows one to sum up the fixed~$\ell$
bounds without any losses in~$\ell$. 

For the pointwise bounds we write \eqref{eq:Yelldecomp} in the form 
\[
 \psi[0](x,\omega) =(\psi_0,\psi_1)(x,\omega)= \sum_{\ell=0}^\infty \sum_{-\ell\le j\le \ell} (f_{\ell,j}(x),g_{\ell,j}(x))Y_{\ell,j}(\omega)
\]
The evolution of these data is given by
\[
 \psi(t,x,\omega) = \sum_{\ell=0}^\infty \sum_{-\ell\le j\le \ell} \psi_{\ell,j}(t,x)Y_{\ell,j}(\omega)
\]
where $\psi_{\ell,j}$ is the evolution of $(f_{\ell,j},g_{\ell,j})$ under \eqref{eq:ellwave}. Therefore, setting $f_{\ell,j}=0$ for ease of notation, 
and using the bound $\|Y_{\ell,j}\|_\infty\les \ell^{\frac12}$ yields 
\begin{align*}
 &\| \la x\ra^{-4}  \psi(t) \|_{L^\infty(\R;L^\infty(S^2))} \\ &\les \sum_{\ell=0}^\infty \la \ell\ra^{\frac12} \sum_{-\ell\le j\le \ell} \| \la x\ra^{-4} \psi_{\ell,j}(t)\|_{L^\infty_x} \\ 
&\les \la t\ra^{-3} \sum_{\ell=0}^\infty \la \ell\ra^{\frac{11}{2}} \sum_{-\ell\le j\le \ell} \|  \la x\ra^{4}  g_{\ell,j}(x) \|_{L^1_x}  \\
&\les \la t\ra^{-3} \sum_{\ell=0}^\infty \la \ell\ra^{-\frac{7}{2}} \sum_{-\ell\le j\le \ell} \| \la x\ra^{4}\la (-\Delta_S)^{\frac{9}{2}}\psi_0(x,\omega), Y_{\ell,j}(\omega)\ra_{S^2}   \|_{L^1_x} 
\end{align*}
where we invoked the pointwise bounds of Lemma~\ref{lem:largeEdecay}, Corollary~\ref{cor:decay_top}, and Proposition~\ref{prop:displow} for large~$\ell$,
and~\cite{DS} and~\cite{DSS} for the remaining~$\ell$. This bound can now be summed since 
\begin{align*}
 &\sum_{\ell=0}^\infty \la \ell\ra^{-\frac{7}{2}} \sum_{-\ell\le j\le \ell} \| \la x\ra^{4} \la (-\Delta_S)^{\frac{9}{2}}\psi_0(x,\omega), Y_{\ell,j}(\omega)\ra_{S^2}   \|_{L^1_x} \\
 &\les \|\la x\ra^{4} (-\Delta_S)^{\frac{9}{2}}\psi_0(x,\omega)\|_{L^1_{x,\omega}}
\end{align*}
This implies the estimate \eqref{eq:decaywaveL1} and Theorem~\ref{thm:main} is proved.  The proof of Theorem~\ref{thm:main2} is analogous.

\section{Semiclassical Sigal-Soffer propagation estimates}\label{sec:HSS}

In this section we present a semiclassical version of the abstract
 theory from~\cite{HSS}. Our arguments are very close
to~\cite{HSS}, but some care is required in keeping track of powers of~$\hbar$. 
The main result is as follows.  In this section $H$
and $A$ are self-adjoint operators on a Hilbert space. $H=H(\hbar)$
and $A=A(\hbar)$ depend on a small parameter $\hbar\in(0,\hbar_0]$
but with domains independent of~$\hbar$. We assume the bounds
\begin{equation}
  \label{eq:comm_ass} \|\ad_A^k(g(H))\|\le C(k,g) \hbar^k
\end{equation}
for all $k\ge0$, $\hbar\in(0,\hbar_0]$ and smooth, compactly
supported functions~$g$ on the line. As usual, $\ad_A^k(g(H))$ are
the $k$-fold iterated commutators defined inductively as
$\ad_A^1(g(H))=[g(H),A]$ and \[ \ad_A^{k}(g(H))=[\ad_A^{k-1}(g(H)),A] \qquad \forall \;k\ge2.\]

\begin{prop}
   \label{prop:HSS}  Suppose $I\subset \R$ is a compact interval so
   that\footnote{It would be perhaps more natural to expect $\chi_I(H) \frac{i}{\hbar} [H,A] \chi_I(H)\ge
     \theta \chi_I(H)$, see~\cite{Graf}, \cite{HisNak}. The loss of an $\hbar$ in the lower bound is due to the fact that we
     establish
     the Mourre estimate at an energy which {\em is trapping}, namely the top of the potential barrier. }
   \begin{equation}
     \label{eq:mourre} \chi_I(H) \frac{i}{\hbar} [H,A] \chi_I(H)\ge
     \theta\hbar \chi_I(H)
   \end{equation}
for some $\theta>0$. Both $I$ and $\theta$ are independent
of~$\hbar$. Then for any smooth $g_I$ with support in~$I$ one has for all $t\in\R$ 
\begin{equation}
  \label{eq:HSS_decay} \| \la   A\ra^{-\alpha} e^{-i\frac{tH}{\hbar}}
  g_I(H) \la   A\ra^{-\alpha} f\| \le C\, \la \hbar t\ra^{-\alpha} \| f\|
\end{equation}
for any $\alpha\ge0$ where $C$ depends on $\alpha$, $\theta$, $g_I$
and~$I$, but not on~$\hbar$. Moreover, $\hbar_0$ needs to be taken
sufficiently small depending  on these parameters.
\end{prop}

The proof requires some preparatory work. First, recall the
commutator expansion formula going back to~\cite{SigSof1},
\cite{SigSof2}, and subsequently refined in~\cite{HunzSig},
\cite{Skib}, \cite{IvSig}, \cite{AmBouGeo}:
\begin{equation}
\label{eq:comm_exp} [g(H),f(A)] = \sum_{k=1}^{n-1}
\frac{f^{(k)}(A)}{k!}\ad_A^k(g(H)) + R_n
\end{equation}
where $f,g$ are smooth, compactly supported functions on the line
and the error $R_n$ satisfies the bound
\begin{equation}
\|R_n\|\le C_n \|\ad_A^n(g(H))\| \sum_{k=0}^{n+2} \int \la
x\ra^{k-n-1} |f^{(k)}(x)|\, dx \label{eq:err_est}
\end{equation}
with a constant $C_n$ depending {\em only} on~$n\ge1$. This error
bound is obtained by means of the {\em Helffer-Sj\"ostrand} formula
involving almost analytic extensions of~$f$, see~\cite[Chapter~2]{Davies}.  
For the expansion~\eqref{eq:comm_exp} and the error bound~\eqref{eq:err_est}
see Appendix~B in~\cite{HunzSig}, in particular (B.8) and~(B.14). 

 In particular, if $f$ {\em is of order at most $p$}
meaning that $f$ is a smooth function on the line obeying the bound
\[|f^{(k)}(x)|\le C_k \la x\ra^{p-k}\] for each~$k\ge0$,
then~\eqref{eq:comm_exp} can be applied with $n>p$.

We now proceed as in~\cite{HSS}. Throughout this section, the
assumptions of Proposition~\ref{prop:HSS} will be in force.

\begin{lemma}
\label{lem:2.1} Let $f\ge0$ be or order $p<4$, nonincreasing and
with $f(x)=0$ for $x\ge0$. Furthermore, assume that $f=f_2^4$ and
$-f_2f_2'=f_3^2$ where $f_2,f_3$ are smooth.  Let $1\le s<\infty$,
$a\in\R$, $A_s:=(\hbar s)^{-1}(A-a)$, and fix  $\eps\in(0,1]$ as well as
$n\ge2$. Then with $g_I$ as above
\begin{multline}
g_I(H) {i}[H,f(A_s)]g_I(H) \le s^{-1} \hbar  \theta g_I(H) f'(A_s) g_I(H) 
 + s^{-1-\eps} g_I(H) f_1(A_s) g_I(H) \\
 + 
s^{-(2n-1-\eps)} g_I^2(H)  \label{eq:HSS_step1}
\end{multline}
uniformly in $a\in\R$ and $\hbar\in(0,\hbar_0]$. Here, $f_1$ is of the same 
order $p<4$ as $f$, and vanishes on $x\ge0$, and it depends
only on $f,g_I$ and~$n$.
\end{lemma}
\begin{proof}
We replace $H$ with $H_b:=H b(H)$ where $b$ is a smooth cutoff
function with $bg=g$ (for simplicity, we write $g$ instead of $g_{I}$). Then
\[
B_k := i \hbar^{-k}\ad_A^k(H_b),\quad k\ge1
\]
satisfy the bounds $\|B_k\|\le C_k$ by assumption for all $k\ge1$.
We begin by showing that
\begin{equation}\label{eq:claim1}
i[H_b, f(A_s)] \simeq -s^{-1} (-f'(A_s))^{\frac12} B_1
(-f'(A_s))^{\frac12}
\end{equation}
where $\simeq$ throughout this proof will mean equality up to
addition of a quadratic form remainder  $\rem=\rem(s)$ satisfying the bound
\[
\pm \rem(s) \le  s^{-(1+\eps)} f_1(A_s) + s^{-(2n-1-\eps)} \mathrm{Id}
\]
uniformly in $\hbar, a$ and with $f_1$ as above. Clearly, any term
of the form $\rem$ is admissible for the lemma and can be ignored.
Write $f=F^2$ and expand by means of~\eqref{eq:comm_exp}
\begin{align}
i[H_b,f(A_s)] &= i[H_b,F(A_s)]F(A_s) + F(A_s)i[H_b,F(A_s)]  \nn \\
&= \sum_{k=1}^{n-1} \frac{1}{k!}  s^{-k} (F^{(k)}(A_s) B_k F(A_s) + F(A_s)B_k^* F^{(k)}(A_s)) \label{eq:sumkAs} \\
 &\qquad +  s^{-n}  (R_nF(A_s)+F(A_s)R_n^*)   \label{eq:RnFAs}
\end{align}
From \eqref{eq:err_est} and since $n\ge2$ and $F$ is of order $<2$,
one concludes that $R_n$ is bounded uniformly in $s,a,\hbar$. We now 
now claim  that only the term $k=1$ is significant, i.e., 
\begin{equation}\label{eq:reduc1}
i [H_b,f(A_s)] \simeq s^{-1} (F'(A_s)B_1 F(A_s) + F(A_s)B_1
F'(A_s))
\end{equation}
Indeed, we first check that  the terms in~\eqref{eq:sumkAs} for $k\ge2$ are subsumed in
the $f_1$ expression of~$\rem$. To see this, we note that
\begin{multline}\label{eq:psi bd}
| \lan F^{(k)}(A_s) B_k F(A_s) \psi,\psi\ran|  \le \|B_{k}\| \| F(A_{s})\psi\|_{2}\|F^{(k)}(A_s)\psi\|_{2}  \\
\le C \sqrt{ \lan F(A_{s})^{2}\psi, \psi\ran  \lan F^{(k)}(A_{s})^{2}\psi, \psi\ran } \le  \lan f_{1}\psi,\psi\ran 
\end{multline}
provided $f_{1}$ is an upper envelope for both $F^{2}$ and $(F^{(k)})^{2}$ with some multiplicative constant. 
Second, for~\eqref{eq:RnFAs} one uses that
\[
\pm(P^*Q+Q^*P)\le P^*P+Q^*Q
\]
with $Q:=\|R_{n}\|s^{-\frac12(1+\eps)}F(A_{s})$, $P^*:=\|R_{n}\|^{-1}s^{-n+\frac12(1+\eps)}R_{n}$. The $Q^{*}Q$ expression
is again subsumed into~$f_{1}$, whereas for $P^{*}P$ we obtain
\[
\|R_{n}\|^{-2}s^{-2n+1+\eps}R_{n}R_{n}^{*}\le s^{-2n+1+\eps} 
\]
This establishes our claim~\eqref{eq:reduc1}. 
By  assumption we can write $F=u^2$, $-F'=v^2$ with $u,v$ of order $<1$
whence $\|[B_1,u(A_s)]\|\le Cs^{-1}$ and $\|[B_1,v(A_s)]\|\le
Cs^{-1}$. Therefore, the right-hand side of~\eqref{eq:reduc1} is of
the form
\begin{align*}
&s^{-1} (F'(A_s)B_1 F(A_s) + F(A_s)B_1 F'(A_s)) \\
& = -s^{-1}(v^2(A_s) B_1 u^2(A_s) + u^2(A_s) B_1 v^2(A_s)) \\
&\simeq -2s^{-1} uv(A_s) B_1 uv(A_s)
\end{align*}
whence \eqref{eq:claim1} since $f'=-2(uv)^2$. The remainder that arises
here is of the $f_{1}$-form as can be seen by arguing as in~\eqref{eq:psi bd}. 

To invoke the Mourre
estimate~\eqref{eq:mourre}, we choose $G$ smooth and compactly
supported in~$I$ and with $bgG=gG=g$. Then
\[
G(H)B_1G(H)=G(H)\prefac [H,A]G(H)\ge \theta \hbar G^2(H)
\]
We now claim that
\begin{equation}\label{eq:final claim}
s^{-1} G(H) \eta B_1 \eta G(H) \simeq s^{-1} \eta G(H) B_1 G(H) \eta
\end{equation}
where $\eta(A_s):=(-f'(A_s))^{\frac12}$. It is clear that this claim will
finish the proof.  One has
\begin{align*}
&s^{-1} G(H) \eta B_1 \eta G(H)-  s^{-1} \eta G(H) B_1 G(H) \eta \\
&= s^{-1}(\eta G B_1 [\eta,G] + [G,\eta]B_1G\eta +
[G,\eta]B_1[\eta,G])
\end{align*}
and there is the expansion
\[
[G,\eta]=\sum_{k=1}^{n-1} \frac{ s^{-k}}{k!}
\eta^{(k)}(A_s)\,\hbar^{-k}\ad_A^k(G(H)) + s^{-n} R
\]
The expansion of $[\eta,G]$ is the adjoint of this one. To prove~\eqref{eq:final claim},
we observe that $\hbar^{-k}\ad_A^k(G(H))$ is uniformly bounded in~$k$, and we also gain~$s^{-k-1}\le s^{-2}$
with $k\ge1$ since~\eqref{eq:final claim} is of order~$s^{-1}$, and each step in the expansion gains
another~$s^{-1}$.  The other issues, such as the domination by $f_{1}$ etc.~are very similar to what 
we have done before, and we skip them. These details are identical to those in~\cite[Lemma~2.1]{HSS}, see in particular
the paragraph leading up to (2.10) in that reference. 
\end{proof}

The following is the semiclassical analogue in this context of the
key propagation estimate of Theorem~1.1 in~\cite{HSS}.

\begin{lemma}
\label{lem:prop_est} Let $0<\theta'<\theta$ and $g_I$ be as in Proposition~\ref{prop:HSS}.
Let $\chi^{\pm}$ be the indicator functions of $\R^{\pm}$,
respectively. Then for any $t\ge0$ and any $m\ge0$,
\[
\|\chi^-(A-a-\hbar\theta't) e^{-i\frac{Ht}{\hbar}} g_I(H)
\chi^+(A-a)\|\le C_m \,\lan \hbar t\ran^{-m} 
\]
uniformly in $\hbar, a$ where $C_m$ only depends
on $m$, $\theta',\theta$, and $g_I$.
\end{lemma}
\begin{proof}
Define for any $s\ge1$ 
\[
A_{s,t}:= (\hbar s)^{-1}(A-a-\hbar \theta t)
\]
Choose $F\ge0$ smooth, nonincreasing of order $0$ and
$F(x)=0$ for $x\ge0$. We shall prove the estimate 
\begin{equation}
  \label{eq:Ast}
  \|F(A_{s,t}) e^{-i\frac{tH}{\hbar}} g_I(H) \chi^+(A-a)\|\le C\,
   \la \hbar  t\ra^{-m}
\end{equation}
To see that this implies the lemma, set $s=t$ in~\eqref{eq:Ast} and note that if $F=1$ on $(-\infty,-\delta]$ with $\delta>0$ small, then 
\[
F(A_{t,t})\chi^-(A-a-\hbar\theta't)  = \chi^-(A-a-\hbar\theta't),\qquad \theta':=\theta-\delta
\]
Define
\begin{align*}
   \phi_s(t) &:= g_I(H) f(A_{s,t}) g_I(H),\quad f=F^2 \\
   \psi_t &:= e^{-i\frac{tH}{\hbar}}\chi^+(A-a)\phi
\end{align*}
where $\phi$ is an arbitrary unit vector. Then 
\eqref{eq:Ast} will follow from the claim:  for every positive integer~$m$, 
\begin{equation}
  \label{eq:boot}
   \la \phi_s(t)\ra_t := \la \psi_t,\phi_s(t)\psi_t\ra \le C\, \hbar^{-m} s^{-m}
\end{equation}
uniformly in $\hbar, a$ and $0\le t\le s$, $1\le s$. Note that $ \la \phi_s(t)\ra_t\ge0$ by construction.  Fix some~$m$. Differentiating
yields
\begin{align*}
  \del_t\la \phi_s(t)\ra_t &= \la \psi_t,D_t\phi_s(t)\psi_t\ra = \lan D_{t}\phi_{s}\ran_{t} \\
  D_t\phi_s(t) &= \prefac [H,\phi_s(t)] +\del_t \phi_s(t) \\
  &= g_I(H) \prefac[H,f(A_{s,t})]g_I(H) - s^{-1} \theta
  g_I(H)f'(A_{s,t})g_I(H)
\end{align*}
By \eqref{eq:comm_exp}, for any $n\ge1$,
\[
0\le \la \phi_s(0)\ra_0 \le C_n\,  s^{-2n}
\]
The point here is that $f^{(k)}(A_{s,0})\chi^{+}(A-a)=0$ for all $k\ge0$ so that only the
remainder in the commutator expansion contributes. 
Next, apply Lemma~\ref{lem:2.1} with $\eps=1$ to conclude that
\begin{equation}\label{eq:Dt phi bd}
D_t \phi_s(t) \le \hbar^{-1} \big[ s^{-2} g_I(H) f_1(A_{s,t}) g_I(H) + s^{-2(n-1)}
g_I(H)^2 \big]
\end{equation}
where $f_1$ satisfies the same hypotheses as $f$; in particular, it is of order zero (one can choose $F$
above so that the properties of $f=F^2$ required by
Lemma~\ref{lem:2.1} are valid). Moreover, we fixed $n$ much larger than $m$. 
Integrating this bound 
in $0\le t\le s$ therefore implies that
\begin{equation}
\la \phi_s(t)\ra_t \le C ( s^{-2n} + \hbar^{-1}s^{-1}) \label{eq:step0}
\end{equation}
which implies~\eqref{eq:boot} with $m=1$. 
 The idea is now to
bootstrap using~\eqref{eq:Dt phi bd}. 
Indeed,  we can apply~\eqref{eq:step0}
to~$f_1$ to conclude that~\eqref{eq:boot} holds with $m=2$.
Iterating this procedure concludes the proof.
\end{proof}

\begin{proof}[Proof of Proposition~\ref{prop:HSS}]
This follows from Lemma~\ref{lem:prop_est} as follows. Let $t\ge0$. 
First, write 
\[
\la A\ra^{-\alpha} = \la A\ra^{-\alpha} \chi^{+}(A+ \frac12 \hbar \theta t) +
\la A\ra^{-\alpha} \chi^{-}(A+ \frac12 \hbar \theta t) 
\]
The second term satisfies 
\[
\| \la A\ra^{-\alpha} \chi^{-}(A+ \frac12 \hbar \theta t)\|  \le C\, \hbar^{-\alpha} t^{-\alpha}
\]
in the sense of operator norms. The first term we subject to the evolution: with $a=-\frac12 \hbar \theta t$, 
\begin{multline}\nn
e^{-i\frac{Ht}{\hbar}} g_I(H)
\chi^+(A-a) = \chi^-(A-a-\frac34\hbar\theta t) e^{-i\frac{Ht}{\hbar}} g_I(H)
\chi^+(A-a) \\
+  \chi^+(A-a-\frac34\hbar\theta t) e^{-i\frac{Ht}{\hbar}} g_I(H)
\chi^+(A-a)
\end{multline}
The second term here satisfies 
\[
\| \la A\ra^{-\alpha} \chi^+(A-a-\frac34\hbar\theta t) e^{-i\frac{Ht}{\hbar}} g_I(H)
\chi^+(A-a)
\|\le C\, \hbar^{-\alpha} t^{-\alpha} 
\]
while the first satisfies the same bound  without the weights  $\la A\ra^{-\alpha}$ by 
Lemma~\ref{lem:prop_est} which concludes the proof for positive times. For negative times one passes
to the adjoints. 
\end{proof}

\end{document}